\documentclass[a4paper,11pt]{amsart}
\usepackage{amsfonts}
\usepackage{amssymb}
\usepackage{mathrsfs}
\usepackage{amsmath,amssymb,amsthm,latexsym,amscd,mathrsfs}
\usepackage{indentfirst}
 \newcommand{\ROM}[1]{\mathrm{\uppercase\expandafter{\romannumeral#1}}}
  \theoremstyle{definition}
 \newtheorem{thm}{Theorem}[section]
 \newtheorem{lem}{Lemma}[section]
 
 \newtheorem{cor}{Corollary}[section]
  
 \newtheorem{rem}{Remark}[section]
 \newtheorem{prop}{Proposition}[section]

\newtheorem{ack}{Acknowledgements}   

\title[On submanifolds with CMC tubular hypersurfaces]{\textbf{On submanifolds whose tubular hypersurfaces have constant mean curvatures}}
\author[J. Q. Ge]{Jianquan Ge}\address{School of Mathematical Sciences, Laboratory of Mathematics and Complex Systems, Beijing Normal
University, Beijing 100875, China}\email{jqge@bnu.edu.cn}
\thanks {The project is partially supported by the NSFC (No.11001016), the SRFDP, and the Program for Changjiang Scholars and Innovative
Research Team in University.}

\date{}
\begin{document}
\maketitle
\begin{abstract}
Motivated by the theory of isoparametric hypersurfaces, we study
submanifolds whose tubular hypersurfaces have some constant ``higher
order mean curvatures". Here a $k$-th order mean curvature $Q_k$
($k\geq1$) of a hypersurface $M^n$ is defined as the $k$-th power
sum of the principal curvatures, or equivalently, of the shape
operator. Many necessary restrictions involving principal
curvatures, higher order mean curvatures and Jacobi operators on
such submanifolds are obtained, which, among other things,
generalize some classical results in the theory of isoparametric
hypersurfaces given by E. Cartan, K. Nomizu, H. F. M{\"u}nzner, Q.
M. Wang, \emph{etc.}. As an application, we finally get a
geometrical filtration for the focal varieties of isoparametric
functions on a complete Riemannian manifold.
\end{abstract}

\section{Introduction}
A hypersurface $M^n$ of a Riemannian manifold $N^{n+1}$ is called
\emph{isoparametric}, if $M^n$ is locally a regular level set of a
function $f$, so-called \emph{isoparametric function}, with the
property that both $\|\nabla f\|^2$ and $\triangle f$ are constant
on the level sets of $f$. One can show that $M^n$ is an
isoparametric hypersurface of $N^{n+1}$ if and only if its nearby
parallel hypersurfaces have constant mean curvature (see
\cite{Th00}, \cite{Ce} for excellent surveys).

The theory of isoparametric hypersurfaces originated from studies on
hypersurfaces of constant principal curvatures in real space forms.
On this topic, E. Cartan started a series of researches by proving
the following characterization (cf.
\cite{Ca38,Ca39,Ca3,Ca4}):\\
\textbf{Theorem.} (\cite{Ca38}) A hypersurface in a real space form
has constant principal curvatures if and only if its nearby parallel
hypersurfaces have constant mean curvature.

Therefore, a hypersurface in a real space form has constant
principal curvatures if and only if it is isoparametric. This
characterization does not hold in more general ambient spaces; see
\cite{Wa82}, \cite{GTY} and \cite{DD10} where counterexamples are
given in complex projective spaces and complex hyperbolic spaces.
However, under an additional assumption that it is a
curvature-adapted hypersurface, this characterization also holds in
a locally rank one symmetric space as showed in Theorem 1.4 of
\cite{GTY}. Note that in a real space form every hypersurface is
curvature-adapted and thus this assumption is superfluous. In this
paper, by applying the Riccati equation and some algebraic geometry,
we will give a generalization of this characterization for these two
cases by using ``higher order mean curvature" instead of (1st order)
mean curvature; see Theorem \ref{Thm-Hyper-space form} and Theorem
\ref{Thm-Hyper-rank one} later. Here higher order mean curvatures
will be defined by power sum polynomials of the principal curvatures
other than elementary symmetric polynomials as usual.

Notice that parallel hypersurfaces of a hypersurface $M$ can be
looked as half-tubular hypersurfaces of $M$, we turn to consider
submanifolds whose tubular hypersurfaces have some constant mean
curvatures. Recall that in the classical theory of isoparametric
hypersurfaces in unit spheres,  Nomizu \cite{No73} showed that each
compact isoparametric hypersurface is a tubular hypersurface of some
(exactly two) submanifold(s), namely focal submanifold(s), and by
using the constancy of the mean curvature of these tubular
hypersurfaces, he proved that the focal submanifolds are minimal.
Later, as a fundamental step in his remarkable work, M{\"u}nzner
\cite{Mu80} proved that these focal submanifolds have constant
principal curvatures which implies the austerity\footnote{A
submanifold of a Riemannian manifold is called an \emph{austere
submanifold} in the sense of \cite{HL82} if its principal curvatures
in any normal direction occur as pairs of opposite signs.} and also
the minimality of the focal submanifolds. Here we say that a
submanifold of higher codimension has \emph{constant principal
curvatures}, if the set of the eigenvalues of the shape operator
$S_{\nu}$ at any point is independent of the choices of the unit
normal vector $\nu$ and the point of the submanifold. This is
different from that in \cite{BCO03} where the principal curvatures
are constant with respect to a (local) parallel normal vector field
and thus may depend on the choices of unit normal vectors.

When the ambient space is a general complete Riemannian manifold
$N^{n+1}$ and $f$ is a global isoparametric function on $N$, Wang
\cite{Wa87} showed that (1) there are at most two singular level
sets, namely the focal varieties of $f$, and they are submanifolds
(both may be disconnected and of different
dimensions\footnote{Henceforth, a connected component of the focal
varieties of an isoparametric function $f$ on a complete Riemannian
manifold $N$ will be called a \emph{focal submanifold} of $f$.}) of
$N$; (2) each regular level set (isoparametric hypersurface) of $f$
is a tubular hypersurface around either of the focal varieties; (3)
(claimed without proof) the focal varieties are minimal. Based on
the structural results (1-2) for the focal varieties, Wang's claim
(3) just asserts the minimality of submanifolds whose tubular
hypersurfaces have constant (1st order) mean curvature, which
generalizes Nomizu's result to arbitrary Riemannian manifolds (see a
more general result of this form for compact submanifolds in
\cite{MP05}). However, M{\"u}nzner's result mentioned above does not
hold in this general case, but it indeed holds for submanifolds
whose tubular hypersurfaces have constant principal curvatures (and
thus each order mean curvature is constant); see a proof of this
assertion and Wang's claim (3) in \cite{GT10}. In this paper, we
will study submanifolds whose tubular hypersurfaces have some
constant higher order mean curvatures in a general Riemannian
manifold. By some technical treatment for the Taylor expansion
formulae of higher order mean curvatures of the tubular
hypersurfaces, we will show that such submanifolds must have some
higher order mean curvatures and some curvature invariants involving
the Jacobi operator of the ambient space being constant, which in
particular will generalize the results mentioned above given by
\cite{No73}, \cite{Mu80}, \cite{Wa87} and \cite{GT10}; see Theorem
\ref{Thm-subm-l} later. As an application, we finally get a
geometrical filtration for the focal submanifolds of isoparametric
functions on a complete Riemannian manifold according to the
filtration of isoparametric functions introduced by \cite{GTY}; see
Theorem \ref{Thm-focal-filt} later.

To state the theorems explicitly, we have to set up some notations.
First of all, as in \cite{GTY} we denote by $\rho_k$ (resp.
$\sigma_i$) the $k$-th power sum polynomial (resp. the $i$-th
elementary symmetric polynomial) in $n$ variables for $k\geq1$ and
$\rho_0\equiv n$ (resp. $1\leq i\leq n$ and $\sigma_0\equiv1$). For
an $n$ by $n$ real symmetric matrix (or self-dual operator) $A$ with
$n$ real eigenvalues $(\mu_1,\cdots,\mu_n)=:\mu$, we denote by
$\rho_k(A):=tr(A^k)=\rho_k(\mu)$ and $\sigma_i(A)=\sigma_i(\mu)$.

Let $M^m$ be a submanifold of a Riemannian manifold $N^{n+1}$. For
any unit normal vector $\nu\in \mathcal {V}_1M$ (unit normal bundle
of $M$), denote by $S_{\nu}$ the shape operator of $M^m$ in
direction $\nu$. Then for any $k\geq1$, we define the \emph{$k$-th
order mean curvature} $Q_k^{\nu}$ in direction $\nu$ by the $k$-th
power sum polynomial of the shape operator other than the $k$-th
elementary symmetric polynomial as usual, \emph{i.e.},
\begin{equation*}
Q_k^{\nu}:=\rho_k(S_{\nu})=tr((S_{\nu})^k).
\end{equation*}
When $M$ is a hypersurface and $\nu$ is a fixed global unit normal
vector field, we simply write the $k$-th order mean curvature
$Q_k^{\nu}$ by $Q_k$. Recall that in \cite{GTY}, we introduced the
following notions: For $1\leq k\leq n$, a non-constant smooth
function $f$ on a Riemannian manifold $N^{n+1}$ is called
\emph{$k$-isoparametric}, if $\|\nabla f\|^2$ and $\rho_1(H_f),
\rho_2(H_f),\cdots,\rho_k(H_f)$ are constant on the level sets of
$f$,\footnote{In \cite{GTY}, we assumed some smoothness of these
functions as one-parameter functions of $f$ for some regularity
reasons; also see a note given in \cite{GT10}. Anyway, without
confusion, we emphasize the geometrical meaning behind the algebraic
definitions.} where $H_f$ is the Hessian of $f$ on $N^{n+1}$; a
hypersurface $M^n$ of $N^{n+1}$ is called \emph{$k$-isoparametric},
if $M^n$ is locally a regular level set of a $k$-isoparametric
function on $N^{n+1}$; an $n$-isoparametric function (hypersurface)
on $N^{n+1}$ is also called a \emph{totally isoparametric} function
(hypersurface). Note that $1$-isoparametric functions
(hypersurfaces) are just isoparametric functions (hypersurfaces). It
was proved there that $M^n$ is a $k$-isoparametric hypersurface if
and only if its nearby parallel hypersurfaces have constant higher
order mean curvatures $Q_1,Q_2,\cdots,Q_k$. Therefore, the sets of
$1$-, $2$-, $\cdots$, $n$-isoparametric functions (hypersurfaces)
give a filtration for isoparametric functions (hypersurfaces) on a
Riemannian manifold $N^{n+1}$ with the filtered geometrical property
that $1$-isoparametric hypersurfaces have constant ($1$st order)
mean curvature, $2$-isoparametric hypersurfaces have constant $1st$
and $2$nd order mean curvatures, and so on, finally,
$n$-isoparametric hypersurfaces have constant principal curvatures.

Let $R$ be the Riemannian curvature tensor and $\nabla$ the
Levi-Civita connection of $N^{n+1}$. In this paper we use the
curvature convention as the following: for any tangent vectors
(vector fields) $X,Y,Z,W$ of $N^{n+1}$,
\begin{equation*}
R(X,Y,Z)=R_{XY}Z:=(\nabla_{[X,Y]}-[\nabla_{X},\nabla_{Y}])Z,
\end{equation*}
and then the covariant derivative $\nabla R$ is also a tensor field
and can be written as:
\begin{eqnarray}
&&(\nabla
R)(X,Y,Z,W)=(\nabla_{W}R)(X,Y,Z)\nonumber\\&:=&\nabla_{W}(R_{XY}Z)-R_{(\nabla_{W}X)Y}Z-R_{X(\nabla_{W}Y)}Z-R_{XY}(\nabla_{W}Z).\nonumber
\end{eqnarray}
 For any tangent vector $\xi\in\mathcal {T}N$, the \emph{Jacobi operator}
$K_{\xi}:\mathcal {T}N\rightarrow\mathcal {T}N$ of $N^{n+1}$ in
direction $\xi$ is defined by
\begin{equation}\label{defn-Jacobi}
K_{\xi}(X):=R_{\xi X}\xi,\quad for~~X\in\mathcal {T}N.
\end{equation}
Note that by properties of the Riemannian curvature tensor,
$K_{\xi}$ is a self-dual linear operator, $tr(K_{\xi})=Ric(\xi)$ is
just the Ricci curvature in direction $\xi$, and $K_{\xi}=K_{-\xi}$,
$K_{\xi}(\xi)\equiv0$. Then without confusion, we will use the same
symbol $K_{\xi}$ when the Jacobi operator is looked as a self-dual
operator on the subspace $\xi^{\bot}$ normal to $\xi$ in $\mathcal
{T}N$. Recall that a submanifold $M^m$ of $N^{n+1}$ is called
\emph{curvature-adapted} (or \emph{compatible}), if the direct sum
$S_{\nu}\oplus I_{n-m}$ of the shape operator $S_{\nu}$ and the
identity map commutes with the Jacobi operator $K_{\nu}$, or
equivalently, these two self-dual operators are simultaneously
diagonalizable, for any unit normal vector $\nu$ of $M$ (cf.
\cite{Bern91}, \cite{Gr04}).

Corresponding to the decomposition of the tangent bundle $\mathcal
{T}N$ on $M^m$, we would like to decompose the Jacobi operator
$K_{\nu}$ ($\nu\in \mathcal {V}_1M$) into two self-dual linear
operators, say \emph{tangent Jacobi operator}
$K_{\nu}^{\top}:\mathcal {T}M\rightarrow \mathcal {T}M$ and
\emph{vertical Jacobi operator} $K_{\nu}^{\bot}:\mathcal
{V}M\rightarrow\mathcal {V}M$ as the following:
\begin{eqnarray}
&&K_{\nu}^{\top}(X):=\emph{projection to}~\mathcal {T}M~\emph{of}~
K_{\nu}(X),\quad for~~ X\in\mathcal {T}M,\nonumber\\
&&K_{\nu}^{\bot}(\eta):=\emph{projection to}~\mathcal
{V}M~\emph{of}~ K_{\nu}(\eta),\quad\quad for~~ \eta\in\mathcal
{V}M.\nonumber
\end{eqnarray}
Obviously, $K_{\nu}^{\top}$, $K_{\nu}^{\bot}$ are self-dual linear
operators and $K_{\nu}^{\top}=K_{-\nu}^{\top}$,
$K_{\nu}^{\bot}=K_{-\nu}^{\bot}$, $K_{\nu}^{\bot}(\nu)\equiv0$.
Without confusion, we denote by the same symbol $K_{\nu}^{\bot}$
when the vertical Jacobi operator $K_{\nu}^{\bot}$ is restricted to
the subspace $\nu^{\bot}\cap \mathcal {V}M$ in $\mathcal {V}M$. Then
under any orthonormal frame $\{e_1,\cdots,e_{n+1}\}$ of $\mathcal
{T}N$ on $M$ with $e_1,\cdots,e_m$ tangent to $M$ and
$e_{m+1},\cdots,e_n,e_{n+1}=\nu$ normal to $M$, the Jacobi operator
$K_{\nu}$ can be expressed as the following symmetric matrix
\begin{equation}\label{Jacobi matrix}
K_{\nu}=\left(\begin{array}{cc|c}K_{\nu}^{\top}&B_{\nu}&0\\
B_{\nu}^t&K_{\nu}^{\bot}&0\\ \hline0&0&0\end{array}\right),
\end{equation}
 where
$K_{\nu}^{\top}$ and $K_{\nu}^{\bot}$ are the matrix expressions of
the tangent and (restricted) vertical Jacobi operators, $B_{\nu}$ is
an $m$ by $(n-m)$ matrix with the property that $B_{\nu}=B_{-\nu}$.

Finally, for any tangent vector (field) $\xi\in\mathcal {T}N$, we
need to introduce another self-dual linear operator, say
\emph{covariant Jacobi operator} $\mathcal {K}_{\xi}:\mathcal
{T}N\rightarrow\mathcal {T}N$, from covariant derivative of the
Riemannian curvature tensor $R$ as the following:
\begin{equation}\label{defn-cov-Jacobi}
\mathcal {K}_{\xi}(X):=(\nabla
R)(\xi,X,\xi,\xi)=(\nabla_{\xi}R)(\xi,X,\xi),\quad for~~X\in\mathcal
{T}N.
\end{equation}
Note that when $\nabla_{\xi}\xi=0$, $\mathcal
{K}_{\xi}=\nabla_{\xi}K_{\xi}$ is just the covariant derivative of
the Jacobi operator $K_{\xi}$ in direction $\xi$. By properties of
the Riemannian curvature tensor and its covariant derivative, it is
easily seen that $\mathcal {K}_{\xi}$ is a self-dual linear operator
and $\mathcal {K}_{-\xi}=-\mathcal {K}_{\xi}$, $\mathcal
{K}_{\xi}(\xi)\equiv0$. In the same way as the decomposition
(\ref{Jacobi matrix}) of the Jacobi operator $K_{\nu}$ ($\nu\in
\mathcal {V}_1M$), we also decompose the covariant Jacobi operator
$\mathcal {K}_{\nu}$ into two self dual operators, say
\emph{covariant tangent Jacobi operator} $\mathcal {K}_{\nu}^{\top}:
\mathcal {T}M\rightarrow \mathcal {T}M$ and \emph{covariant vertical
Jacobi operator} $\mathcal {K}_{\nu}^{\bot}: \mathcal
{V}M\rightarrow \mathcal {V}M$, as the following:
\begin{eqnarray}
&&\mathcal {K}_{\nu}^{\top}(X):=\emph{projection to}~\mathcal
{T}M~\emph{of}~
\mathcal {K}_{\nu}(X),\quad for~~ X\in\mathcal {T}M,\nonumber\\
&&\mathcal {K}_{\nu}^{\bot}(\eta):=\emph{projection to}~\mathcal
{V}M~\emph{of}~ \mathcal {K}_{\nu}(\eta),\quad\quad for~~
\eta\in\mathcal {V}M.\nonumber
\end{eqnarray}
Obviously, $\mathcal {K}_{\nu}^{\top}$, $\mathcal {K}_{\nu}^{\bot}$
are self-dual linear operators and $\mathcal
{K}_{-\nu}^{\top}=-\mathcal {K}_{\nu}^{\top}$, $\mathcal
{K}_{-\nu}^{\bot}=-\mathcal {K}_{\nu}^{\bot}$, $\mathcal
{K}_{\nu}^{\bot}(\nu)\equiv0$. Without confusion, we denote by the
same symbol $\mathcal {K}_{\nu}^{\bot}$ when the covariant vertical
Jacobi operator $\mathcal {K}_{\nu}^{\bot}$ is restricted to the
subspace $\nu^{\bot}\cap \mathcal {V}M$ in $\mathcal {V}M$. Under
the same orthonormal frame as in (\ref{Jacobi matrix}), the
covariant Jacobi operator can be expressed as the following
symmetric matrix
\begin{equation}\label{cov Jacobi matrix}
\mathcal {K}_{\nu}=\left(\begin{array}{cc|c}\mathcal {K}_{\nu}^{\top}&\mathcal {B}_{\nu}&0\\
\mathcal {B}_{\nu}^t&\mathcal {K}_{\nu}^{\bot}&0\\
\hline0&0&0\end{array}\right),
\end{equation}
 where
$\mathcal {K}_{\nu}^{\top}$ and $\mathcal {K}_{\nu}^{\bot}$ are the
matrix expressions of the covariant tangent and (restricted)
covariant vertical Jacobi operators, $\mathcal {B}_{\nu}$ is an $m$
by $(n-m)$ matrix with the property that $\mathcal
{B}_{-\nu}=-\mathcal {B}_{\nu}$.

Now we are ready to state the theorems. Firstly, for the
hypersurface case, by applying the Riccati equation and some
algebraic geometry, we obtain the following generalizations of
Cartan's Theorem (\cite{Ca38}) and Theorem 1.4 of \cite{GTY},
respectively.
\begin{thm}\label{Thm-Hyper-space form}
A hypersurface in a real space form has constant principal
curvatures if and only if for some $k\geq1$, its nearby parallel
hypersurfaces have constant $k$-th order mean curvature $Q_k$.
\end{thm}
\begin{thm}\label{Thm-Hyper-rank one}
A curvature-adapted hypersurface in a locally rank one symmetric
space has constant principal curvatures if and only if for some
$k\geq1$, its nearby parallel hypersurfaces have constant $k$-th
order mean curvature $Q_k$.
\end{thm}
\begin{rem}
In these cases, the hypersurface is totally isoparametric. On the
other hand, it is still unknown that whether a hypersurface of
constant principal curvatures in a locally rank one symmetric space
other than real space form is totally isoparametric.
\end{rem}

For general submanifolds, by some technical treatment for the Taylor
expansion formulae of higher order mean curvatures of the tubular
hypersurfaces, we obtain the following generalizations of those
results on geometry of the focal submanifolds in the theory of
isoparametric hypersurfaces given by \cite{No73}, \cite{Mu80},
\cite{Wa87} and \cite{GT10}.
\begin{thm}\label{Thm-subm-l}
Let $M^m$ be a submanifold of a Riemannian manifold $N^{n+1}$.
Suppose that on any nearby tubular hypersurface $M_t^n$ of $M^m$ in
$N^{n+1}$ ($t\in(0,\varepsilon)$),
\begin{itemize}
\item[(a)] for $1\leq l\leq 4$, the $l$-th, $(l+1)$-th, $\cdots$, $(l+[\frac{l}{2}])$-th order mean
curvatures $Q_{l},Q_{l+1},\cdots,Q_{l+[\frac{l}{2}]}$ are constant;
\item[(b)] for $l\geq5$, the $l$-th, $(l+1)$-th, $\cdots$, $(l+[\frac{l}{2}]+1)$-th order mean
curvatures $Q_{l},Q_{l+1},\cdots,Q_{l+[\frac{l}{2}]+1}$ are
constant.
\end{itemize}
Then the $l$-th order mean curvature $Q_{l}^{\nu}$ in any direction
$\nu$ is a constant independent of the choices of the unit normal
vector $\nu$ and the point of $M^m$, and so are the curvature
invariants: $\rho_{[\frac{l}{2}]}(K_{\nu}^{\bot})$ when $l$ is even;
and $tr\Big((K_{\nu}^{\bot})^{[\frac{l}{2}]-1}\mathcal
{K}_{\nu}^{\bot}\Big)$ when $l\geq3$ is odd. In particular, when $l$
is odd, $Q_{l}^{\nu}\equiv0$ and
$tr\Big((K_{\nu}^{\bot})^{[\frac{l}{2}]-1}\mathcal
{K}_{\nu}^{\bot}\Big)\equiv0$. Furthermore, if in addition we assume
the constancy of $Q_{l-1}$ in (a) and (b) for $l\geq2$, then we have
a new constant curvature invariant
$tr\Big((S_{\nu})^{l-2}K_{\nu}^{\top}\Big)$ which also vanishes when
$l$ is odd.
\end{thm}
\begin{rem}
As indicated by the theorem, if we assume more constant higher order
mean curvatures on tubular hypersurfaces, we would possibly get more
constant curvature invariants such as
$tr\Big((K_{\nu}^{\bot})^{i}(\mathcal {K}_{\nu}^{\bot})^{j}\Big)$,
$tr\Big((S_{\nu})^{i}(K_{\nu}^{\top})^j\Big)$ and even
$tr\Big((S_{\nu})^{i}(\mathcal {K}_{\nu}^{\top})^{j}\Big)$ on the
submanifold, though the computations would be rather complicated.
See Theorem \ref{Thm-subm-d} for a more detailed description.
\end{rem}

At last, we conclude this section by the following geometrical
filtration for the focal submanifolds of isoparametric functions on
a complete Riemannian manifold according to the filtration of
isoparametric functions by $1$-,$2$-,$\cdots$,$n$-isoparametric
functions.
\begin{thm}\label{Thm-focal-filt}
Let $M^m$ be a focal submanifold of an isoparametric function $f$ on
a complete Riemannian manifold $N^{n+1}$. Suppose that $f$ is a
$k$-isoparametric function for some $1\leq k\leq n$.
\begin{itemize}
\item[(i)] If $k=1$, then for any unit normal vector
$\nu$ on $M$,
\begin{equation*}
Q_1^{\nu}\equiv0, \quad
Q_2^{\nu}+tr(K_{\nu}^{\top})+\frac{1}{3}tr(K_{\nu}^{\bot})\equiv
Const\footnote{Throughout this paper, ``$\equiv Const$" for
something involving $\nu$ means that it is a constant independent of
the choices of the unit normal vector $\nu\in\mathcal {V}_1M$ and
the point of the submanifold $M$.},
\end{equation*}
\begin{equation*}
Q_3^{\nu}+tr(S_{\nu}K_{\nu}^{\top})+\frac{1}{2}tr(\mathcal
{K}_{\nu}^{\top})+\frac{1}{4}tr(\mathcal {K}_{\nu}^{\bot})\equiv 0,
\end{equation*}
in particular, $M$ is a minimal submanifold in $N$;
\item[(ii)] If $k=2$, then besides the identities in (i), we have further
\begin{equation*}
Q_2^{\nu}-\frac{2}{3}tr(K_{\nu}^{\bot})\equiv Const,\quad
Q_3^{\nu}+tr(S_{\nu}K_{\nu}^{\top})-\frac{1}{4}tr(\mathcal
{K}_{\nu}^{\bot})\equiv 0,
\end{equation*}
and thus $tr(K_{\nu})=Ric(\nu)\equiv Const, \quad tr(\mathcal
{K}_{\nu})\equiv0;$
\item[(iii)] If $k=3$, then besides the identities in (i-ii), we have
further
\begin{equation*}
tr(K_{\nu}^{\bot})\equiv Const,\quad
Q_3^{\nu}+\frac{3}{4}tr(\mathcal {K}_{\nu}^{\bot})\equiv 0,
\end{equation*}
and thus $Q_2^{\nu}\equiv Const, \quad
tr(S_{\nu}K_{\nu}^{\top})-tr(\mathcal {K}_{\nu}^{\bot})\equiv 0;$
\item[(iv)] If $k=4$, then besides the identities in (i-iii), we have
further
\begin{equation*}
tr(\mathcal {K}_{\nu}^{\bot})\equiv 0,
\end{equation*}
and thus $Q_3^{\nu}\equiv 0, \quad tr(S_{\nu}K_{\nu}^{\top})\equiv
0;$
\item[(v)] If $k=5$, then besides the identities in (i-iv), we have
further
\begin{equation*}
Q_{4}^{\nu}-\frac{2}{9}\rho_2(K_{\nu}^{\bot})\equiv Const, \quad
3tr\Big(S_{\nu}^{2}K_{\nu}^{\top}\Big)+\rho_2(K_{\nu}^{\bot})\equiv
Const;\end{equation*}
\item[(vi)] If $k=6$, then besides the identities in (i-v), we have
further
\begin{equation*}
\rho_2(K_{\nu}^{\bot})\equiv Const,\quad
2tr\Big(S_{\nu}^{3}K_{\nu}^{\top}\Big)+3Q_{5}^{\nu}+\frac{1}{12}tr\Big(K_{\nu}^{\bot}\mathcal
{K}_{\nu}^{\bot}\Big)\equiv0;
\end{equation*}
and thus $Q_4^{\nu}\equiv Const,\quad
tr(S_{\nu}^2K_{\nu}^{\top})\equiv Const;$
\item[(vii)] If $k=3d+1,d\geq2$, then besides the identities for $k\leq 3d$, we have further
\begin{eqnarray}
&&Q_{2d}^{\nu}\equiv Const,\quad \rho_d(K_{\nu}^{\bot})\equiv Const,
\quad tr(S_{\nu}^{2d-2}K_{\nu}^{\top})\equiv Const,\nonumber\\
&&Q_{2d+1}^{\nu}-d~3^{-d+1}4^{-1}tr\Big((K_{\nu}^{\bot})^{d-1}\mathcal
{K}_{\nu}^{\bot}\Big)\equiv 0,\nonumber\\
&&(2d)tr\Big(S_{\nu}^{2d-1}K_{\nu}^{\top}\Big)+(3d+1)Q_{2d+1}^{\nu}\equiv
0;\nonumber
\end{eqnarray}
\item[(viii)] If $k=3d+2,d\geq2$, then besides the identities in (i-vii), we have further
\begin{eqnarray}
&&Q_{2d+1}^{\nu}\equiv 0,\quad
tr\Big(S_{\nu}^{2d-1}K_{\nu}^{\top}\Big)\equiv0,\quad
tr\Big((K_{\nu}^{\bot})^{d-1}\mathcal {K}_{\nu}^{\bot}\Big)\equiv
0,\nonumber\\
&&(2d+1)tr\Big(S_{\nu}^{2d}K_{\nu}^{\top}\Big)+(3d+2)Q_{2d+2}^{\nu}-3^{-d-1}\rho_{d+1}(K_{\nu}^{\bot})\equiv
Const;\nonumber
\end{eqnarray}
\item[(ix)] If $k=3d+3,d\geq2$, then besides the identities in (i-viii), we have further
\begin{eqnarray}
&&Q_{2d+2}^{\nu}+3^{-d-1}\rho_{d+1}(K_{\nu}^{\bot})\equiv Const; \nonumber\\
&&2tr\Big(S_{\nu}^{2d+1}K_{\nu}^{\top}\Big)+3Q_{2d+3}^{\nu}+3^{-d}4^{-1}tr\Big((K_{\nu}^{\bot})^{d}\mathcal
{K}_{\nu}^{\bot}\Big)\equiv0.\nonumber
\end{eqnarray}
\end{itemize}
Furthermore,
\begin{itemize}
\item[(a)] if $m=0$, \emph{i.e.}, $M$ is a point, then the Ricci
curvature of $N$ is constant on $M$;
\item[(b)] if $m=n$, \emph{i.e.}, $M$ is a hypersurface,
 then $M$ is a $k$-isoparametric hypersurface
with $Q_1=Q_3=\cdots=Q_{2j+1}\equiv0$ for $2j+1\leq k$;
\item[(c)] if $m\leq[\frac{2k+1}{3}]$ for $k\leq6$, or $m\leq[\frac{2k-1}{3}]$ for $k\geq7$, or $k=n$, then $M$ is an austere submanifold
of constant principal curvatures in $N$, if in addition
$m=2,~n\geq4$; or $m=3,~ n\geq5$; or $m=4,~n\geq10$, then $M$ is a
totally geodesic submanifold;
\item[(d)] if $m\geq n-[\frac{k}{3}]$ for $k\leq6$, or $m\geq n-[\frac{k-1}{3}]$ for $k\geq7$, or $k=n$, then the
vertical Jacobi operator $K_{\nu}^{\bot}$ has constant eigenvalues
independent of the choices of the unit normal vector $\nu$ and the
point of $M$, or equivalently, the restriction of the Riemannian
curvature model $(\mathcal {T}N,R)$ of $N$ to the normal bundle of
$M$ is an Osserman curvature model.
\end{itemize}
\end{thm}
\begin{rem}
Recall that \cite{Ge} introduced $(2j)$-th mean curvature function
$K_{2j}$ and $(2j+1)$-th mean curvature vector field $H_{2j+1}$ on a
submanifold $M^m$ of a Riemannian manifold $N^{n+1}$ which
generalize higher order mean curvature functions on a hypersurface,
and showed that $K_{2j}$ (resp. $H_{2j+1}$) equals, up to a constant
factor, the integral of $\sigma_{2j}(S_{\nu})$ (resp.
$\nu\sigma_{2j+1}(S_\nu)$) over the unit normal sphere of $M$ in
$N$. Consequently, by Newton's identities, $K_{l}\equiv Const$ ($l$
even) and $H_{l}\equiv0$ ($l$ odd) on the focal submanifold $M^m$ of
a $k$-isoparametric function with $l\leq[\frac{2k+1}{3}]$ for
$k\leq6$, or $l\leq[\frac{2k-1}{3}]$ for $k\geq7$.
\end{rem}
\begin{rem}
An isoparametric function is called a \emph{properly isoparametric}
function if the focal submanifolds have codimension greater than $1$
(cf. \cite{GT09}). So $f$ in case (b) is not properly isoparametric
and since in this case the focal submanifolds or their normal line
bundles could be non-orientable, there may be no global unit normal
vector fields, in which case the conclusion in (b) should be
considered as local property on $M$. Note that classically
isoparametric hypersurfaces in unit spheres are assumed to be
connected and thus the focal submanifolds have codimension greater
than one.
\end{rem}
\begin{rem}
It was proved by Chi \cite{Chi1} and Nikolayevsky \cite{N1,N2} that
an Osserman curvature model of dimension $q\neq16$ is isomorphic to
one of the curvature models given by Clifford module structures (see
a detailed introduction in \cite{BGN}). In particular, if $q$ is
odd, then the Jacobi operator of an Osserman curvature model has
only one constant eigenvalue except the trivial eigenvalue $0$. So
if $n-m$ is even in (d), then the restricted vertical Jacobi
operator $K_{\nu}^{\bot}\equiv Const\cdot id$ and thus the sectional
curvatures of $N$ in normal planes of $M$ are constant.
\end{rem}

\section{Shape operators of tubular hypersurfaces}\label{sec shape
op}
In this section, by using the Fermi coordinates, we will mainly
derive a Taylor expansion formula up to order $2$ about $t$ of the
shape operator $S(t)$ of the tubular hypersurface $M_t^n$ of radius
$t\in(0,\varepsilon)$ around a submanifold $M^m$. This Taylor
expansion formula up to order $1$ has been given in \cite{GT10}.

Let $M^m$ be a submanifold of a Riemannian manifold $N^{n+1}$ and
$M_t^n$ be the tubular hypersurface around $M$ of sufficiently small
radius $t\in(0,\varepsilon)$. Then the ``outward" unit normal vector
field $\nu_t$ of $M_t$ for $t\in(0,\varepsilon)$ forms a unit vector
field, say $\xi$, on an open subset $\mathcal
{N}_{\varepsilon}M:=\bigcup_{t\in(0,\varepsilon)}M_t$ of $N^{n+1}$.
The shape operator $S(t)$ of $M_t$ with respect to $\nu_t$ at a
point $q\in M_t$ is just the restriction to $M_t$ of the tensorial
operator $S:\mathcal {T}(\mathcal {N}_{\varepsilon}M)\rightarrow
\mathcal {T}(\mathcal {N}_{\varepsilon}M)$ defined by
\begin{equation}\label{shape def}
S(X):=-\nabla_X\xi,
\end{equation}
 for $X\in \mathcal {T}_{q}(\mathcal {N}_{\varepsilon}M)$, where $\nabla$ denotes the covariant derivative in $N$. It is
easily seen that $S$ is self-dual and $S(\xi)=0$. Taking covariant
derivative of $S$ with respect to $\xi$ gives the well-known Riccati
equation (cf. \cite{Gr04}):
\[\nabla_{\xi}S=S^2+K_{\xi},\]
and its restriction to $M_t$ can be written as
\begin{equation}\label{Riccati eq}
S'(t)=S(t)^2+R(t),
\end{equation}
where $S'(t):=(\nabla_{\xi}S)|_{\mathcal
{T}M_t}=(\nabla_{\nu_t}S)|_{\mathcal {T}M_t}$, $K_{\xi}$ is the
Jacobi operator of $N$ in direction $\xi$ defined in
(\ref{defn-Jacobi}) and $R(t):=K_{\xi}|_{\mathcal
{T}M_t}=K_{\nu_t}|_{\mathcal {T}M_t}$.

Now we choose a system of Fermi coordinates in a neighborhood
$\widetilde{\mathcal {U}}$ of any point $p\in M$ in $N$ as follows
(cf. \cite{GT10}). First we choose normal geodesic coordinates
$(y_1,\cdots,y_m)$ centered at $p$ in a neighborhood $\mathcal {U}$
of $p$ in $M$. Then in $\mathcal {U}$ we fix orthonormal sections
$E_{m+1},\cdots,E_n,E_{n+1}$ of the normal bundle $\mathcal {V}M$ of
$M$ in $N$ such that they are parallel with respect to the normal
connection along any geodesic ray from $p$ in $M$ and
$E_{n+1}|_p=\nu$ for a given unit normal vector $\nu$ of $M$ at $p$.
The Fermi coordinates $(x_1,\cdots,x_n,x_{n+1})$ of $(\mathcal
{U}\subset M\subset)$ $\widetilde{\mathcal {U}}\subset N$ centered
at $p$ are defined by
\begin{eqnarray}
&&x_a\Big(exp_{q}\Big(\sum_{j=m+1}^{n+1}t_jE_j(q)\Big)\Big)=y_a(q)
\quad \quad (a=1,\cdots,m),\nonumber\\
&&x_i\Big(exp_{q}\Big(\sum_{j=m+1}^{n+1}t_jE_j(q)\Big)\Big)=t_i
\quad \quad (i=m+1,\cdots,n+1),\nonumber
\end{eqnarray}
for $q\in\mathcal {U}$ and any sufficiently small numbers
$t_{m+1},\cdots,t_{n+1}$ with $\sum_{i}t_i^2< \varepsilon^2$. Then
the \emph{Generalized Gauss Lemma} shows that in
$\widetilde{\mathcal {U}}-M\subset\mathcal {N}_{\varepsilon}M$,
\begin{equation}\label{Gauss lemma}
\xi=\sum_{i=m+1}^n\frac{x_i}{\sigma}\partial x_i=\nabla\sigma,
\end{equation}
where $\sigma:=\sqrt{\sum_{i}x_i^2}$ is the distance function to $M$
in $\widetilde{\mathcal {U}}$ (cf. \cite{Gr04}).
 It follows from the
definition that along the normal geodesic
$\eta_{\nu}(t):=exp_p(t\nu)$ in $\widetilde{\mathcal {U}}$,
\begin{equation}\label{patial xn+1}
\partial
x_{n+1}|_{\eta_{\nu}(t)}=\eta_{\nu}'(t)=\nu_t=\xi|_{\eta_v(t)},
\quad for~~t\in(0,\varepsilon).
\end{equation}
Moreover, the coordinate vector fields $\partial x_1,\cdots,\partial
x_{n+1}$ satisfy
\begin{eqnarray}\label{Fermi coord}
\nabla_{\partial x_a}\partial x_b|_p\in\mathcal {V}_pM, \quad \quad
\nabla_{\partial x_a}\partial x_i|_p\in\mathcal {T}_pM,\quad \quad\nabla_{\partial x_i}\partial x_j|_{\mathcal {U}}=0,\label{covar-deriv}\\
\langle\partial x_{\alpha}, \partial
x_{\beta}\rangle|_p=\delta_{\alpha\beta}, \quad \quad
\langle\partial x_{a}, \partial x_{i}\rangle|_\mathcal {U}=0,\quad
\quad \langle\partial x_{i}, \partial x_{j}\rangle|_\mathcal
{U}=\delta_{ij}, \nonumber
\end{eqnarray}
where $\langle,\rangle$ denotes the metric, and the indices
convention is that indices $a,b,\cdots\in\{1,\cdots,m\}$, indices
$i,j,\cdots\in\{m+1,\cdots,n+1\}$ and indices
$\alpha,\beta,\cdots\in\{1,\cdots,n+1\}$. Then $\partial
x_1|_p,\cdots,\partial x_m|_p$ form an orthonormal frame of
$\mathcal {T}_pM$ and $\partial x_{m+1}|_p,\cdots,\partial
x_{n+1}|_p=\nu$ form an orthonormal frame of $\mathcal {V}_pM$, and
under these frames, the Jacobi operator $K_{\nu}$, the covariant
Jacobi operator $\mathcal {K}_{\nu}$ of $N$ can be written as real
symmetric matrices as (\ref{Jacobi matrix}) and (\ref{cov Jacobi
matrix}) respectively.

 Now in $\widetilde{\mathcal
{U}}-M\subset\mathcal {N}_{\varepsilon}M$, we express the self-dual
operator $S$ defined by (\ref{shape def}) as a real matrix
$S=(S_{\alpha\beta})$ (not symmetric in general) of order $n+1$
under the coordinate vector fields $\partial x_1,\cdots,\partial
x_{n+1}$, \emph{i.e.}, $S(\partial
x_{\alpha}):=\sum_{\beta=1}^{n+1}S_{\alpha\beta}\partial x_{\beta}$.
By properties of the Fermi coordinates, in \cite{GT10} we obtained
the following expansion formula of $S$.
\begin{prop}(cf. \cite{GT10}) With notations as above,
at the point $\eta_{\nu}(t)=exp_p(t\nu)\in M_t$ for any $t\in
(0,\varepsilon)$, the following expansion formula holds
\begin{eqnarray}\label{shape expansion}
S&=&\left(\begin{array}{ccc}S_{\nu}+t(S_{\nu}^2+K_{\nu}^{\top})+\mathcal{O}(t^2)&
tB_{\nu}+\mathcal {O}(t^2)& \mathcal{O}(t^2)
\\\frac{t}{3}B^t_{\nu}+\mathcal {O}(t^2)&-\frac{1}{t}I+\frac{t}{3}K_{\nu}^{\bot}+\mathcal {O}(t^2)& \mathcal{O}(t^2)
\\0& 0& 0 \end{array}\right),
\end{eqnarray}
where $S_{\nu}:=(h^{\nu}_{ab})$ is the matrix of the shape operator
of $M$ in direction $\nu$ under the orthonormal frame $\partial
x_1|_p,\cdots,\partial x_m|_p$, $\mathcal {O}(t^2)$ denotes matrices
with elements of $t$'s order not less than $2$.
\end{prop}

Rewrite the expansion formula (\ref{shape expansion}) by power
series about $t$ as:
\begin{equation}\label{shape power series}
S=\sum_{r=0}^{\infty}t^{r-1}S_r=\frac{1}{t}S_0+S_1+tS_2+t^2S_3+\mathcal{O}(t^3),
\end{equation}
where $$S_0=\left(\begin{array}{cc|c}0&0&0\\
 0&-I&0\\\hline 0&0&0\end{array}\right),\quad S_1=\left(\begin{array}{cc|c}S_{\nu}&0&0\\
 0&0&0\\\hline 0&0&0\end{array}\right),\quad S_2=\left(\begin{array}{cc|c}S_{\nu}^2+K_{\nu}^{\top}&B_{\nu}&0\\
 \frac{1}{3}B^t_{\nu}&\frac{1}{3}K_{\nu}^{\bot}&0\\\hline 0&0&0\end{array}\right),$$ and $S_r, r\geq3,$ are matrices independent of $t$.

 To calculate the coefficient matrix $S_3$ of $t^2$ for
this expansion formula, we need the following lemmas.
\begin{lem}(cf. \cite{GT10}, \cite{MMP06})\label{metric expa}
Let $g_{\alpha\beta}:=\langle\partial x_{\alpha},\partial
x_{\beta}\rangle$ and $G:=(g_{\alpha\beta})$ be the matrix of the
metric. Then at the point $\eta_{\nu}(t)=exp_p(t\nu)\in M_t$, we
have
\begin{eqnarray}
&&g_{ab}(t)=\delta_{ab}-2h^{\nu}_{ab}t+\Big(\sum\limits_ch^{\nu}_{ac}h^{\nu}_{cb}-\langle
R_{\nu\partial x_a}\nu,\partial
x_b\rangle\Big)t^2+O(t^3),\nonumber\\
&&g_{ai}(t)=-\frac{2}{3}\langle R_{\nu\partial x_a}\nu,\partial
x_i\rangle t^2+O(t^3),\nonumber\\
&&g_{ij}(t)=\delta_{ij}-\frac{1}{3}\langle R_{\nu\partial
x_i}\nu,\partial x_j\rangle t^2+O(t^3),\nonumber
\end{eqnarray}
 or in matrix form,
 \begin{equation*}
G(t)=I+t\left(\begin{array}{cc|c}-2S_{\nu}&0&0\\
 0&0&0\\\hline 0&0&0\end{array}\right)+t^2\left(\begin{array}{cc|c}S_{\nu}^2-K_{\nu}^{\top}&-\frac{2}{3}B_{\nu}&0\\
-\frac{2}{3}B_{\nu}^t&-\frac{1}{3}K_{\nu}^{\bot}&0\\
\hline 0&0&0\end{array}\right)+\mathcal {O}(t^3).
\end{equation*}
\end{lem}
\begin{lem}(cf. \cite{GT10})\label{covar of coord fields}
Put $\displaystyle\nabla_{\partial x_{\alpha}}\partial
x_{n+1}:=\sum_{\beta}w_{\alpha\beta}\partial x_{\beta}$ and
$W:=(w_{\alpha\beta})$. Then at the point
$\eta_{\nu}(t)=exp_p(t\nu)\in M_t$, we have
\begin{eqnarray}
\nabla_{\partial x_a}\partial
x_{n+1}&=&-\sum\limits_bh^{\nu}_{ab}\partial
x_b-t\sum\limits_b\Big(\sum\limits_ch^{\nu}_{ac}h^{\nu}_{cb}+\langle
R_{\nu\partial x_a}\nu,\partial x_b\rangle\Big)\partial
x_b\nonumber\\&&-t\sum\limits_k\langle R_{\nu\partial
x_a}\nu,\partial x_k\rangle\partial
x_k+\sum\limits_{\alpha}O(t^2)_{\alpha}\partial x_{\alpha},\nonumber\\
\nabla_{\partial x_i}\partial
x_{n+1}&=&-\frac{t}{3}\sum\limits_b\langle R_{\nu\partial
x_i}\nu,\partial x_b\rangle\partial
x_b-\frac{t}{3}\sum\limits_k\langle R_{\nu\partial x_i}\nu,\partial
x_k\rangle\partial x_k+\sum\limits_{\alpha}O(t^2)_{\alpha}\partial
x_{\alpha},\nonumber
\end{eqnarray}
or in matrix form,
\begin{equation*}
W(t)=\left(\begin{array}{cc|c}-S_{\nu}&0&0\\
 0&0&0\\\hline 0&0&0\end{array}\right)+t\left(\begin{array}{cc|c}-S_{\nu}^2-K_{\nu}^{\top}&-B_{\nu}&0\\
-\frac{1}{3}B_{\nu}^t&-\frac{1}{3}K_{\nu}^{\bot}&0\\
\hline 0&0&0\end{array}\right)+\mathcal {O}(t^2).
\end{equation*}
\end{lem}
\begin{lem}\label{3-cov}
At the point $p\in M$, we have
\begin{eqnarray}
&&\nabla_{\nu}\nabla_{\partial x_{n+1}}\nabla_{\partial
x_{n+1}}\partial x_a=-\nabla_{\nu}(R_{\partial x_{n+1}
\partial x_{a}}\partial x_{n+1})=-\mathcal {K}_{\nu}(\partial
x_{a})-K_{\nu}(\nabla_{\nu}\partial x_{a});\nonumber\\
&&\nabla_{\nu}\nabla_{\partial x_{n+1}}\nabla_{\partial
x_{n+1}}\partial x_i=-\frac{1}{2}\nabla_{\nu}(R_{\partial x_{n+1}
\partial x_{i}}\partial x_{n+1})=-\frac{1}{2}\mathcal {K}_{\nu}(\partial
x_{i});\nonumber\\
&&\nabla_{\nu}\nabla_{\partial x_{n+1}}\partial x_a=-R_{\nu\partial
x_a}\partial x_{n+1}=-K_{\nu}(\partial x_a);\nonumber\\
&&\nabla_{\nu}\nabla_{\partial x_{n+1}}\partial
x_i=-\frac{1}{3}R_{\nu\partial x_i}\partial
x_{n+1}=-\frac{1}{3}K_{\nu}(\partial x_i).\nonumber
\end{eqnarray}
\end{lem}
\begin{proof}
It follows from (\ref{patial xn+1}) that along the geodesic
$\eta_{\nu}(t)=exp_p(t\nu)$, $\nabla_{\partial x_{n+1}}\partial
x_{n+1}=\nabla_{\xi}\xi=0$ and thus $\mathcal
{K}_{\xi}=\nabla_{\xi}K_{\xi}$ by definition
(\ref{defn-cov-Jacobi}). Then the first equalities in the identities
follow from Lemma 9.19 and Lemma 9.20 in \cite{Gr04} and the second
equalities follow immediately from (\ref{defn-Jacobi}),
(\ref{defn-cov-Jacobi}) and (\ref{Fermi coord}).
\end{proof}
\begin{lem}\label{W2}
Let $\displaystyle W(t):=\sum_{r=0}^{\infty}t^rW_r$ be the matrix in
Lemma \ref{covar of coord fields} with $W_r$s independent of $t$.
Then
\begin{equation*}
W_2=\left(\begin{array}{cc|c}-\frac{1}{2}\mathcal {K}_{\nu}^{\top}-
S_{\nu}^3-K_{\nu}^{\top}S_{\nu}&-\frac{1}{2}\mathcal {B}_{\nu}&0\\
-\frac{1}{4}\mathcal {B}_{\nu}^t-\frac{1}{3}B_{\nu}^tS_{\nu}&-\frac{1}{4}\mathcal {K}_{\nu}^{\bot}&0\\
\hline 0&0&0\end{array}\right).
\end{equation*}
\end{lem}
\begin{proof}
Put $u_{\alpha\beta}:=\langle\nabla_{\partial x_{\alpha}}\partial
x_{n+1},\partial x_{\beta}\rangle$ and
$U(t):=(u_{\alpha\beta})|_{\eta_{\nu}(t)}$. Then $U(t)=W(t)G(t),$
and thus
\begin{equation*}
U''(0)=W''(0)G(0)+2W'(0)G'(0)+W(0)G''(0),
\end{equation*}
where $G(t)$ is the
matrix in Lemma \ref{metric expa} and so $G(0)=I$,
\begin{equation}\label{G0}
G'(0)=\left(\begin{array}{cc|c}-2S_{\nu}&0&0\\
 0&0&0\\\hline 0&0&0\end{array}\right),\quad
 G''(0)=2\left(\begin{array}{cc|c}S_{\nu}^2-K_{\nu}^{\top}&-\frac{2}{3}B_{\nu}&0\\
-\frac{2}{3}B_{\nu}^t&-\frac{1}{3}K_{\nu}^{\bot}&0\\
\hline 0&0&0\end{array}\right),
\end{equation}
and by Lemma \ref{covar of coord fields},
\begin{equation}\label{W0}
W(0)=\left(\begin{array}{cc|c}-S_{\nu}&0&0\\
 0&0&0\\\hline 0&0&0\end{array}\right),\quad W'(0)=\left(\begin{array}{cc|c}-S_{\nu}^2-K_{\nu}^{\top}&-B_{\nu}&0\\
-\frac{1}{3}B_{\nu}^t&-\frac{1}{3}K_{\nu}^{\bot}&0\\
\hline 0&0&0\end{array}\right).
\end{equation}
On the other hand,
\begin{eqnarray}
u_{\alpha\beta}''(0)&=&\nu\nu_t\langle\nabla_{\partial
x_{\alpha}}\partial x_{n+1},\partial x_{\beta}\rangle=\nu\partial
x_{n+1}\langle\nabla_{\partial
x_{\alpha}}\partial x_{n+1},\partial x_{\beta}\rangle\nonumber\\
&=&\langle\nabla_{\nu}\partial
x_{\alpha},\nabla_{\nu}\nabla_{\partial x_{n+1}}\partial
x_{\beta}\rangle+2\langle \nabla_{\nu}\nabla_{\partial
x_{n+1}}\partial x_{\alpha},\nabla_{\nu}\partial
x_{\beta}\rangle\nonumber\\&&+\langle \nabla_{\nu}\nabla_{\partial
x_{n+1}}\nabla_{\partial x_{n+1}}\partial x_{\alpha},\partial
x_{\beta}\rangle,\nonumber
\end{eqnarray}
then by Lemma \ref{3-cov}, we can get
\begin{equation*}
U''(0)=\left(\begin{array}{cc|c}-\mathcal {K}_{\nu}^{\top}+2K_{\nu}^{\top}S_{\nu}+2S_{\nu}K_{\nu}^{\top}&-\mathcal {B}_{\nu}+\frac{4}{3}S_{\nu}B_{\nu}&0\\
-\frac{1}{2}\mathcal {B}_{\nu}^t+\frac{2}{3}B_{\nu}^tS_{\nu}&-\frac{1}{2}\mathcal {K}_{\nu}^{\bot}&0\\
\hline 0&0&0\end{array}\right).
\end{equation*}
Combining the above formulae, we can get the required formula for
$W_2$ by the following
\begin{equation*}
W_2=\frac{1}{2}W''(0)=\frac{1}{2}\Big(U''(0)-2W'(0)G'(0)-W(0)G''(0)\Big).
\end{equation*}
\end{proof}

\begin{cor}
Let $S_3$ be the coefficient matrix of $t^2$ in (\ref{shape power
series}). Then
\begin{equation}\label{shape S3}
S_3=\left(\begin{array}{cc|c}\frac{1}{2}\mathcal {K}_{\nu}^{\top}+
S_{\nu}^3+K_{\nu}^{\top}S_{\nu}&\frac{1}{2}\mathcal {B}_{\nu}&0\\
\frac{1}{4}\mathcal {B}_{\nu}^t+\frac{1}{3}B_{\nu}^tS_{\nu}&\frac{1}{4}\mathcal {K}_{\nu}^{\bot}&0\\
\hline 0&0&0\end{array}\right).
\end{equation}
\end{cor}
\begin{proof}
It follows from (\ref{shape def}), (\ref{Gauss lemma}) and
(\ref{patial xn+1}) that at the point $\eta_{\nu}(t)=exp_p(t\nu)\in
M_t$,
\begin{eqnarray}
&&-S(\partial x_a)=\nabla_{\partial x_a}\xi=\nabla_{\partial
x_a}\Big(\sum_{j}\frac{x_j}{\sigma}\partial
x_j\Big)=\sum_{j}\frac{x_j}{\sigma}\nabla_{\partial x_a}\partial
x_j=\nabla_{\partial x_a}\partial x_{n+1};\nonumber\\
&&-S(\partial x_i)=\nabla_{\partial x_i}\xi=\nabla_{\partial
x_i}\Big(\sum_{j}\frac{x_j}{\sigma}\partial
x_j\Big)=\sum_{j}\partial x_i\Big(\frac{x_j}{\sigma}\Big)\partial
x_j+\sum_{j}\frac{x_j}{\sigma}\nabla_{\partial x_i}\partial
x_j\nonumber\\&&\quad\quad\quad\quad=\frac{1}{t}\partial
x_i-\frac{1}{t}\delta_{i~n+1}~\partial x_{n+1}+\nabla_{\partial
x_i}\partial x_{n+1},\nonumber
\end{eqnarray}
which, together with Lemma \ref{covar of coord fields} and Lemma
\ref{W2}, gives the required formula for $S_3$ immediately.
\end{proof}

Finally we conclude this section by the following expansion formula
for the Jacobi operator $K_{\xi}$.
\begin{cor}
At the point $\eta_{\nu}(t)=exp_p(t\nu)\in M_t$, we have
\begin{equation*}
K_{\xi}=K_{\nu_t}=K_{\nu}+t\left(\begin{array}{cc|c}\mathcal {K}_{\nu}^{\top}
-S_{\nu}K_{\nu}^{\top}+K_{\nu}^{\top}S_{\nu}&\mathcal {B}_{\nu}-S_{\nu}B_{\nu}&0\\
\mathcal {B}_{\nu}^t+B_{\nu}^tS_{\nu}&\mathcal {K}_{\nu}^{\bot}&0\\
\hline 0&0&0\end{array}\right)+\mathcal {O}(t^2).
\end{equation*}
\end{cor}
\begin{proof}
Obviously it suffices to verify the coefficient matrix, say $K_1$,
of $t$ in this expansion formula for $K_{\xi}$. Firstly by
(\ref{patial xn+1}) and (\ref{defn-cov-Jacobi}), we know that
$\nabla_{\xi}\xi\equiv0$ and thus $\mathcal
{K}_{\xi}=\nabla_{\xi}K_{\xi}$. Then by the Taylor expansion
formula, we calculate the coefficient matrix $K_1$ as follows:\vskip
0.2 cm $\nu\langle K_{\xi}(\partial x_{\alpha}),~\partial
x_{\beta}\rangle=\langle \nabla_{\xi}(K_{\xi}(\partial
x_{\alpha})),\partial x_{\beta}\rangle|_p+\langle K_{\xi}(\partial
x_{\alpha}),\nabla_{\xi}\partial x_{\beta}\rangle|_p$\vskip 0.2 cm
\quad \quad\quad\quad\quad\quad\quad\quad $=\langle
(\nabla_{\xi}K_{\xi})(\partial
x_{\alpha})+K_{\xi}(\nabla_{\xi}\partial x_{\alpha}),\partial
x_{\beta}\rangle|_p+\langle K_{\xi}(\partial
x_{\alpha}),\nabla_{\xi}\partial x_{\beta}\rangle|_p$\vskip 0.2 cm
\quad \quad\quad\quad\quad\quad\quad\quad$=\langle \mathcal
{K}_{\xi}(\partial x_{\alpha}),\partial x_{\beta}\rangle|_p+\langle
\nabla_{\xi}\partial x_{\alpha},K_{\xi}(\partial
x_{\beta})\rangle|_p+\langle \nabla_{\xi}\partial
x_{\beta},K_{\xi}(\partial x_{\alpha})\rangle|_p$,\vskip 0.2 cm or
in matrix form,\vskip 0.2 cm  $(\nu\langle K_{\xi}(\partial
x_{\alpha}),~\partial x_{\beta}\rangle)=\mathcal
{K}_{\nu}+U(0)K_{\nu}+K_{\nu}U(0)^t$,\vskip 0.2 cm where $U(0)=W(0)$
as in Lemma \ref{W2}; on the other hand,\vskip 0.2 cm $(\nu\langle
K_{\xi}(\partial x_{\alpha}),~\partial
x_{\beta}\rangle)=\frac{d}{dt}\Big|_{t=0}\Big(K_{\nu_t}G(t)\Big)=K_1+K_{\nu}G'(0)$;\vskip
0.2 cm and therefore, \vskip 0.2 cm $K_1=\mathcal
{K}_{\nu}+W(0)K_{\nu}+K_{\nu}W(0)^t-K_{\nu}G'(0)$,\vskip 0.2 cm
which gives the required formula by using (\ref{Jacobi matrix}),
(\ref{cov Jacobi matrix}), (\ref{G0}) and (\ref{W0}).
\end{proof}

\section{Hypersurface case}
In this section, we deal with the hypersurface case in our subject
by proving Theorem \ref{Thm-Hyper-space form} and Theorem
\ref{Thm-Hyper-rank one}. Throughout this paper, unless stated
otherwise, notations will be consistent with those in previous
sections.

Firstly we establish a lemma on algebraic geometry which will be
useful in the proof of the theorems.
\begin{lem}\label{AG lem}For each $m,n\geq1$,
define polynomials $P_k\in\mathbb{C}[x_1,\cdots,x_n]$ by
$$P_k:=\rho_{k}(x_1,\cdots,x_n)+\widetilde{P}_{k-1}(x_1,\cdots,x_n),\quad
for ~~k=m,m+1,\cdots,m+n-1,$$ where $\rho_{k}$ is the $k$-th power
sum polynomial, $\widetilde{P}_{k-1}$ is an arbitrary polynomial of
degree less than $k$. Then $P_m,P_{m+1},\cdots,P_{m+n-1}$ form a
regular sequence in $\mathbb{C}[x_1,\cdots,x_n]$. Consequently, the
dimension of each variety $V_k$ in $\mathbb{C}^n$ defined by
$P_{m}=P_{m+1}=\cdots=P_{m+k-1}=0$ is less than or equal to $n-k$
for $k=1,\cdots,n$. In particular, $V_{n}$ is a finite subset of
$\mathbb{C}^n$.
\end{lem}
\begin{proof}
The proof is similar to that of Lemma 4.4 in \cite{GTY}. For
completeness, we repeat it as follows.

 Firstly recall (cf. \cite{E},
\cite{Mat}) that a sequence $r_1,\cdots,r_k$ in a commutative ring
$\mathcal {R}$ with identity is called a \emph{regular sequence} if
$(1)$ the ideal $(r_1,\cdots,r_k)\neq \mathcal {R}$; $(2) $ $r_1$ is
not a zero divisor in $\mathcal {R}$; and $(3)$ $r_{i+1}$ is not a
zero divisor in the quotient ring $\mathcal {R}/(r_1,\cdots,r_i)$
for $i=1,\cdots,k-1$.

Now we will work on the polynomial ring $\mathcal
{R}=\mathbb{C}[x_1,\cdots,x_n]$. Obviously, it is a
\emph{Cohen-Macaulay} ring, possessing the property that
$dim(\mathcal {R}/(r_1,\cdots,r_k))=n-k$ for a regular sequence
$r_1,\cdots,r_k$ in  $\mathcal {R}$. Meanwhile, we know that
$dim(V_k)=dim(\mathcal {R}/I(V_k))$, where $I(V_k)\supset
(r_1,\cdots,r_k)$ is the ideal of the variety $$V_k:=\{x\in
\mathbb{C}^n|r_1(x)=\cdots=r_k(x)=0\}.$$ Therefore, when
$r_1,\cdots,r_n$ form a regular sequence, $dim(V_k)\leq n-k$ for
$k=1,\cdots,n$. In particular, $dim(V_n)=0$. The last assertion in
the lemma is due to the facts that every variety in $\mathbb{C}^n$
can be expressed as a union of finite irreducible varieties and that
a zero-dimensional irreducible variety in $\mathbb{C}^n$ is just a
point. So it suffices to show that the polynomials
$P_m,P_{m+1}\cdots,P_{m+n-1}$ form a regular sequence in $\mathcal
{R}$.

Obviously, $P_m$ forms a regular sequence in $\mathcal {R}$. Suppose
that $P_m,P_{m+1},\cdots,P_{m+n-1}$ do not form a regular sequence,
there exists some $k$ with $1\leq k<n$ such that $P_{m+k}$ is a zero
divisor modulo $(P_m,\cdots,P_{m+k-1})$ in $\mathcal {R}$. Then we
may choose a relation of minimal degree of the form
\begin{equation}\label{minimal relation}
f_mP_m+f_{m+1}P_{m+1}+\cdots+f_{m+k}P_{m+k}=0,
\end{equation}
 where
$f_m,\cdots,f_{m+k}$ are polynomials of minimal degrees modulo
$(P_m,\cdots,P_{m+k-1})$. Denote by $D(>0)$ the maximal degree of
$f_kP_k$'s. Let $f_{i_1}P_{i_1},\cdots,f_{i_r}P_{i_r}$ be those of
maximal degree $D$ for some $m\leq i_1<\cdots<i_r\leq m+k$. Then one
can pick out the homogeneous components
$\tilde{f}_{i_1}\rho_{i_1},\cdots,\tilde{f}_{i_r}\rho_{i_r}$ of
maximal degree from them in equation (\ref{minimal relation}) such
that
\begin{equation}\label{maximal degree}
\tilde{f}_{i_1}\rho_{i_1}+\cdots+\tilde{f}_{i_r}\rho_{i_r}=0,
\end{equation}
where $\tilde{f}_{i_1},\cdots,\tilde{f}_{i_r}$ are the homogeneous
components of maximal degrees of $f_{i_1},\cdots,f_{i_r}$,
respectively. Recall a recent result showed in \cite{CKW} that the
power sum polynomials $\rho_m,\rho_{m+1},\cdots,\rho_{m+n-1}$ form a
regular sequence in $\mathcal {R}$. Then by (\ref{maximal degree}),
$r>1$ and
$\tilde{f}_{i_r}\in(\rho_m,\rho_{m+1},\cdots,\rho_{i_r-1})$, which
imply that there exist homogeneous polynomials
$a_m,a_{m+1},\cdots,a_{i_r-1}$ such that
$$\tilde{f}_{i_r}=a_m\rho_m+a_{m+1}\rho_{m+1}+\cdots+a_{i_r-1}\rho_{i_r-1},$$
and therefore,
$$f_{i_r}=a_mP_m+a_{m+1}P_{m+1}+\cdots+a_{i_r-1}P_{i_r-1}+\hat{f}_{i_r}\equiv \hat{f}_{i_r}, \quad mod~~ (P_m,\cdots,P_{m+k-1}),$$
where $\hat{f}_{i_r}$ is a polynomial of degree less than
$D-i_r=deg(f_{i_r})$, which contradicts the original choice of
minimal relation (\ref{minimal relation}).

The proof is now complete.
\end{proof}

Let $M^n$ be a curvature-adapted hypersurface in a real space form
or locally rank one symmetric space $N^{n+1}$. Denote by $M_t$,
$t\in(-\varepsilon,\varepsilon)$, nearby parallel hypersurfaces of
$M_0=M$ and $\nu_t$ the unit normal vector field on $M_t$. As is
well known, a hypersurface in a real space form is always
curvature-adapted as mentioned in the introduction, and moreover,
its parallel hypersurfaces have common principal eigenvectors up to
parallel translations along normal geodesics, which is a nice
property also preserved by a curvature-adapted hypersurface in a
symmetric space in which case parallel hypersurfaces are still
curvature-adapted (cf.\cite{Gr04}). Now in both cases, the Jacobi
operator $K_{\xi}$ of $N$ has constant eigenvalues independent of
the choices of the unit tangent vector $\xi$ and the point of $N$.
Therefore, one can choose the principal orthonormal eigenvectors
$\{e_i(t)|i=1,\cdots,n\}$ of $M_t$ such that they are parallel along
normal geodesics and simultaneously diagonalize the shape operator
$S(t)$ of $M_t$ and the restricted Jacobi operator
$R(t):=K_{\nu_t}|_{\mathcal {T}_{M_t}}$ of $N$ as the following
symmetric matrices:
\begin{equation*}
S(t)=diag(\mu_1(t),\cdots,\mu_n(t)),\quad
R(t)=diag(\kappa_1,\cdots,\kappa_n),
\end{equation*}
where $\mu_i(t)$'s are principal curvature functions of $M_t$,
$\kappa_i\equiv c$ for $M$ in a real space form with constant
sectional curvature $c$, or $\kappa_i\in\{c,4c\}$ for $M$ in a
locally rank one symmetric space with non-constant sectional
curvature. Moreover, since $\nabla_{\nu_t}e_i(t)=0$,
\begin{equation*} S'(t):=\nabla_{\nu_t}S(t)=diag(\mu_1'(t),\cdots,\mu_n'(t)),
\end{equation*}
and thus the Riccati equation (\ref{Riccati eq}) can be written as
\begin{equation}\label{Riccati eq for curv-adapted}
\mu_i'(t)=\mu_i(t)^2+\kappa_i, \quad for~~i=1,\cdots,n.
\end{equation}

 We are now ready to prove the theorems for the hypersurface case in
 our subject.

\textbf{Proof of Theorem \ref{Thm-Hyper-space form} and Theorem
\ref{Thm-Hyper-rank one}.} Recall that the $k$-th order mean
curvature $Q_k(t)$ of $M_t$ is defined by the $k$-th power sum
polynomial of the principal curvatures $\mu_1(t),\cdots,\mu_n(t)$
for any $k\geq1$, \emph{i.e.},
\begin{equation}\label{Qkt}
Q_k(t)=tr\Big(S(t)^k\Big)=\sum_{i=1}^n\mu_i(t)^k=\rho_k(\mu_1(t),\cdots,\mu_n(t)).
\end{equation}
Taking derivative of $Q_k(t)$ with respect to $t$ by applying
(\ref{Riccati eq for curv-adapted}), we get
\begin{equation}\label{Qkt 1deriv}
\frac{1}{k}Q_k'(t)=\sum_{i=1}^n\Big(\mu_i(t)^{k+1}+\kappa_i~\mu_i(t)^{k-1}\Big)=Q_{k+1}+\sum_{i=1}^n\kappa_i~\mu_i(t)^{k-1}.
\end{equation}
Similarly, for any $j\geq1$, taking the $j$-th derivative of
$Q_k(t)$ with respect to $t$ by applying (\ref{Riccati eq for
curv-adapted}), we can get
\begin{equation}\label{Qkt jderiv}
\frac{1}{k(k+1)\cdots(k+j-1)}Q_k^{(j)}(t)=Q_{k+j}+\widehat{P}_{k+j-1}(\mu_1(t),\cdots,\mu_n(t)),
\end{equation}
where $\widehat{P}_{k+j-1}$ is some polynomial of degree less than
$k+j$ with constant coefficients in $n$ variables.

Now assume that for some $k\geq1$, $Q_k(t)$ is constant on $M_t$ for
any $t\in(-\varepsilon,\varepsilon)$ and thus it is a smooth
function depending only on $t$, so are the derived functions
$Q_k^{(j)}(t)$ for all $j\geq1$. Then for any fixed
$t\in(-\varepsilon,\varepsilon)$, (\ref{Qkt}-\ref{Qkt jderiv}) show
that the principal curvatures $(\mu_1(t),\cdots,\mu_n(t))$ of $M_t$
are solutions of the algebraic equations
\begin{equation}\label{Pk}
P_{l}(x_1,\cdots,x_n):=\rho_{l}(x_1,\cdots,x_n)+\widetilde{P}_{l-1}(x_1,\cdots,x_n)=0,\quad
for~~l=k,k+1\cdots,
\end{equation}
where $\rho_l$ is the $l$-th power sum polynomial,
$\widetilde{P}_{l-1}=\widehat{P}_{l-1}-\frac{(k-1)!}{(l-1)!}Q_k^{(l-k)}(t)$
is a polynomial of degree less than $l$ with constant coefficients.
In particular, $(\mu_1(t),\cdots,\mu_n(t))$ belongs to the variety
$V_n$ in $\mathbb{C}^n$ defined by $P_k=P_{k+1}=\cdots=P_{k+n-1}=0$.
Therefore, by Lemma \ref{AG lem}, we know that
$(\mu_1(t),\cdots,\mu_n(t))$ belongs to a finite subset of
$\mathbb{C}^n$ and thus $\mu_i(t)$'s are constant on $M_t$ since
$M_t$ is connected. It means that $M$ has constant principal
curvatures and is totally isoparametric.

Conversely, if $M$ has constant principal curvatures, by the Riccati
equation (\ref{Riccati eq for curv-adapted}), we know immediately
(cf. \cite{GTY}) that $\mu_i(t)$'s are constant on $M_t$ and so each
order mean curvatures are constant on $M_t$.

The proof is now complete. \hfill $\Box$

\section{General Submanifold case}\label{section 4}
In this section, we deal with the general submanifold case in our
subject. Firstly, by using the Taylor expansion formula of the shape
operator obtained in section \ref{sec shape op}, we derive a power
series expansion formula for higher order mean curvatures of tubular
hypersurfaces around a submanifold in a general Riemannian manifold.
Then through some involved calculations and technical treatments of
this formula, we obtain Theorem \ref{Thm-subm-d} and give a proof of
Theorem \ref{Thm-subm-l}.

As in section \ref{sec shape op}, let $M^m$ be a submanifold of a
Riemannian manifold $N^{n+1}$ and $M_t^n$ be the tubular
hypersurface around $M$ of sufficiently small radius
$t\in(0,\varepsilon)$. Since the shape operator $S(t)$ of $M_t$ is
the restriction of the operator $S$ defined in (\ref{shape def}) to
$\mathcal {T}M_t$ and $S(\nu_t)=S(\xi)|_{M_t}=0$, it follows that
$S(t)$ has the same nonzero eigenvalues as $S$. Therefore, the
$k$-th order mean curvature $Q_k(t)$ of $M_t$ can be calculated by
\begin{equation*}
Q_k(t)=tr(S(t)^k)=tr(S^k)=tr(\widehat{S}(t)^k),
\end{equation*}
where $\widehat{S}(t)$ is the left-up $n$ by $n$ submatrix of $S$ in
(\ref{shape expansion}) and by (\ref{shape power series}),
(\ref{shape S3}), at the point $\eta_{\nu}(t)=exp_p(t\nu)\in M_t$
for any $t\in (0,\varepsilon)$, we have the following expansion
formula
\begin{equation}\label{shape bar expansion}
t\widehat{S}(t)=-A_0+\sum_{r=1}^{\infty}A_rt^r,
\end{equation}
where
\begin{eqnarray}
&&A_0=\left(\begin{array}{cc} 0&0\\  0&I
 \end{array}\right),\quad A_2=\left(\begin{array}{cc}S_{\nu}^2+K_{\nu}^{\top}&B_{\nu}\\
 \frac{1}{3}B^t_{\nu}&\frac{1}{3}K_{\nu}^{\bot}\end{array}\right),
\nonumber\\
&&A_1=\left(\begin{array}{cc} S_{\nu}&0\\  0&0
 \end{array}\right),\quad
A_3=\left(\begin{array}{cc}\frac{1}{2}\mathcal {K}_{\nu}^{\top}+
S_{\nu}^3+K_{\nu}^{\top}S_{\nu}&\frac{1}{2}\mathcal {B}_{\nu}\\
 \frac{1}{4}\mathcal
{B}_{\nu}^t+\frac{1}{3}B_{\nu}^tS_{\nu}&\frac{1}{4}\mathcal
{K}_{\nu}^{\bot}\end{array}\right),\nonumber
\end{eqnarray}
and $A_r$, $r\geq4$, are $n$ by $n$ matrices independent of $t$.

 Put
$\Upsilon_i(t):=t^iQ_i(t)=tr\Big((t\widehat{S}(t))^i\Big)=\displaystyle\sum_{r=0}^{\infty}\Upsilon_{ir}t^r,$
$i\geq1$. Then by comparing the coefficient of $t^r$ in the extended
formula for $\Upsilon_i(t)$ with (\ref{shape bar expansion})
substituted, we get
\begin{equation}\label{Upsilon ir}
\Upsilon_{ir}=\sum_{\sigma\in\mathscr{P}_{ir}}(-1)^{\sigma_0}\sum_{\tau\in\mathscr{S}_i(\sigma)}tr(A_{\tau_1}A_{\tau_2}\cdots
A_{\tau_i}), \quad for~~i\geq1,~~r\geq0,
\end{equation}
where
\begin{equation*}
\mathscr{P}_{ir}:=\Big\{\sigma=(\sigma_0,\sigma_1,\cdots,\sigma_r)\in
\mathbb{Z}^{r+1}~\Big|~\sum_{s=0}^r\sigma_s=i,~\sum_{s=0}^rs\sigma_s=r,~\sigma_s\geq0\Big\};
\end{equation*}
and for
$\sigma=(\sigma_0,\sigma_1,\cdots,\sigma_r)\in\mathscr{P}_{ir}$,
$\Sigma(\sigma):=\{s~|~0\leq s\leq r,~~\sigma_s>0\}$,
\begin{equation*}
\mathscr{S}_i(\sigma):=\Big\{\tau=(\tau_1,\cdots,\tau_i)\in
\mathbb{Z}^{i}~|~\forall~ s\in\Sigma(\sigma), \exists~
\sigma_s~elements~ of~ \tau ~equal~s.\Big\}.
\end{equation*}
Obviously,
\begin{equation}\label{Upsilon 0,1}
\begin{array}{lll}
\Upsilon_{i0}=(-1)^itr(A_0^i)=(-1)^i(n-m),&
\Upsilon_{i1}=\delta_{i1}tr(A_1),& for~~i\geq1,\\
\Upsilon_{1r}=tr(A_r),&& for~~r\geq1.
\end{array}
\end{equation}
From now on, we assume $i\geq2,~r\geq2$ and put
$$d(i,r):=\min\Big\{[\frac{i}{2}],[\frac{r}{2}]\Big\},\quad
D(i,r):=\min\Big\{i-1,[\frac{r}{2}]\Big\}.$$ Note that
$A_0A_1=A_1A_0=0$ and $tr(CD)=tr(DC)$, then
$tr(A_{\tau_1}A_{\tau_2}\cdots A_{\tau_i})=0$ if $0$ and $1$ occur
in some successive indices $\tau_j,\tau_{j+1}$ ($j~mod~i$). Since
$A_0^2=A_0$, we can reduce the sequence in non-vanishing
$tr(A_{\tau_1}A_{\tau_2}\cdots A_{\tau_i})$s such that $A_0$ occurs
separately. Moreover, one can see that $A_0$ occurs separately at
most $d(i,r)$ times with some $A_s$ or $A_sA_1^{\iota_s}A_{s'}$
($s,s'\geq2$, $\iota_s\geq1$) between, in fact, each non-vanishing
$tr(A_{\tau_1}A_{\tau_2}\cdots A_{\tau_i})$,
$\tau\in\mathscr{S}_i(\sigma),~
\sigma\in\mathscr{P}_{ir}$~with~$\sigma_0>0$ can be written as
\begin{equation}\label{trace mon.}
tr(A_0^{\lambda_0}\cdot\widetilde{A}_{s_1}A_0^{\lambda_1}\cdot\widetilde{A}_{s_2}A_0^{\lambda_2}
\cdot\cdots\cdot\widetilde{A}_{s_c}A_0^{\lambda_{c}}),
\end{equation}
where $\widetilde{A}_{s}=A_{s}$ or $A_sA_1^{\iota_s}A_{s'}$ with
$s,s'\geq2$, $\iota_s\geq1$, the sum of $\lambda_s$ equals
$\sigma_0$, the sum of $\iota_s$ equals $\sigma_1$, and $1\leq c\leq
D(i,r)$. According to these, we will refine the summation
(\ref{Upsilon ir}) or actually the index sets $\mathscr{P}_{ir}$ and
$\mathscr{S}_i(\sigma)$ as follows.

For $0\leq a\leq i-1$, denote by $\mathscr{T}_{ir}( a)$ the set of
all non-vanishing $tr(A_{\tau_1}A_{\tau_2}\cdots A_{\tau_i})$s,
$\tau\in\mathscr{S}_i(\sigma)$, $\sigma\in\mathscr{P}_{ir}$ with
$\sigma_0=a$, where the elements with different indices are looked
as different though they may have the same value. For $1\leq c\leq
D(i,r)$, put
\begin{equation*}
\mathcal {A}_{ir}^c:=\Big\{a\in\mathbb{Z}^{+}~|~\exists~
elements~in~\mathscr{T}_{ir}( a)~ of~the~form~(\ref{trace
mon.})\Big\}.
\end{equation*}
Then by straightforward calculations, we get
\begin{equation*}
\mathcal {A}_{ir}^c=\begin{cases}
\{1\},& i=2, ~c=1,~r\geq2,\\
\{i-c\},& i\geq3,~2c\leq r\leq2c+2,\\
\{a~|~i-r+c+1\leq a\leq i-c-2,~ a\geq1\}\cup \{i-c\},& i\geq3,~r\geq
2c+3.
\end{cases}
\end{equation*}
It follows immediately from the definitions that $\mathscr{T}_{ir}(
a)$ is empty if $a>0$ is not in $\mathcal {A}_{ir}^c$ for any $1\leq
c\leq D(i,r)$. For example, $\mathscr{T}_{33}(1)$ is empty since now
each $tr(A_{\tau_1}A_{\tau_2}A_{\tau_3})$ equals
$tr(A_0A_1A_2)=tr(A_1A_0A_2)=0$. On the other hand, putting
\begin{eqnarray}
&&\Lambda^c(a):=\Big\{\lambda=(\lambda_0,\lambda_1,\cdots,\lambda_c)\in\mathbb{Z}^{c+1}~|~
\sum_{s=0}^c\lambda_s= a,~\lambda_s\geq0\Big\},\nonumber\\
&&\mathscr{T}_{ir}(\lambda):=\Big\{\omega\in\mathscr{T}_{ir}(
a)~|~\omega ~is~of~the~form~ (\ref{trace
mon.})~with~(\lambda_0,\cdots,\lambda_c)=\lambda\Big\},~~
for~~\lambda\in\Lambda^c( a),\nonumber
\end{eqnarray}
we can divide $\mathscr{T}_{ir}(a)$ ($a>0$) into subsets
$\mathscr{T}_{ir}(\lambda)$, $\lambda\in\Lambda^c( a)$, $1\leq c\leq
D(i,r)$. Note that each $\mathscr{T}_{ir}(\lambda)$ should be empty
if $\mathscr{T}_{ir}(a)$ is empty. Furthermore, for $1\leq b\leq c$,
putting
\begin{eqnarray}
&&\Lambda^c_b:=\Big\{\mu=(\mu_1,\cdots,\mu_b)~|~1\leq
\mu_1<\cdots<\mu_b< c\Big\},\nonumber\\
&&\overline{\Lambda}^c_b:=\Big\{\mu=(\mu_1,\cdots,\mu_b)~|~1\leq
\mu_1<\cdots<\mu_{b-1}<\mu_b=c\Big\},\nonumber
\end{eqnarray}
and for $\mu\in\Lambda^c_b$,
\begin{equation*}
\Lambda^c( a,\mu):=\Big\{\lambda\in\Lambda^c( a)~|~
\sum_{s=1}^b\lambda_{\mu_s}= a,~\lambda_{\mu_s}\geq1\Big\},
\end{equation*}
and for $\mu\in\overline{\Lambda}^c_b$,
\begin{equation*}
\Lambda^c( a,\mu):=\Big\{\lambda\in\Lambda^c( a)~|~
\lambda_0+\lambda_c+\sum_{s=1}^{b-1}\lambda_{\mu_s}=
a,~\lambda_0+\lambda_c\geq1,~\lambda_{\mu_s}\geq1~for~s<b\Big\},
\end{equation*}
we can divide the index set $\Lambda^c(a)$ into subsets $\Lambda^c(
a,\mu)$, $\mu\in\Lambda^c_b$ or $\overline{\Lambda}^c_b$, $1\leq
b\leq c$. Since the number of elements of $\Lambda^c( a,\mu)$ is
independent of the choices of $\mu\in\Lambda^c_b$ (resp.
$\mu\in\overline{\Lambda}^c_b$) and $c$, we denote it by
$\Theta(a,b)$ (resp. $\overline{\Theta}( a,b)$) which would be zero
for $b> a$. In fact, by a detailed study of the definitions, it
turns out that the $2$-parameter function $\Theta$ satisfies the
inductive relation
\begin{equation}\label{inductive Theta}
\Theta( a+1,b)-\Theta( a,b)=\Theta( a,b-1),
\end{equation}
with initial conditions $\Theta( a,b)=0$ for $ a<b$ or $b<0$,
$\Theta( a,1)=1$ ($\Theta( a,0):\equiv0$) for $ a\geq1$, and so does
the $2$-parameter function $\overline{\Theta}$ with initial
conditions $\overline{\Theta}( a,b)=0$ for $ a<b$ or $b<0$,
$\overline{\Theta}( a,1)= a+1$ ($\overline{\Theta}( a,0):\equiv 1$)
for $ a\geq1$. Notice that for each $\mu\in \Lambda^c_b$ or
$\overline{\Lambda}^c_b$, respectively, each element $\lambda$ of
the index subset $\Lambda^c( a,\mu)$ corresponds to the same subset
$\mathscr{T}_{ir}(\lambda)$, denoted by $\mathscr{T}_{ir}(a,\mu)$
which is non-empty for $a\in\mathcal {A}_{ir}^c$\footnote{Then in
this case, $2b\leq b+c\leq a+c\leq i$ and so $b\leq
\min\{[\frac{i}{2}],[\frac{r}{2}]\}=:d(i,r)$ is just the number of
copies of $A_0$ occurring separately in non-vanishing
$tr(A_{\tau_1}A_{\tau_2}\cdots A_{\tau_i})$s with some
$\widetilde{A}_{s}$ between.}, whose elements have values of the
following form:
\begin{equation*}
tr\Big((\widetilde{A}_{s_1}\cdots
\widetilde{A}_{s_{\mu_1}})A_0\cdot(\widetilde{A}_{s_{\mu_1+1}}\cdots\widetilde{A}_{s_{\mu_2}})
A_0 \cdot\cdots\cdot(\widetilde{A}_{s_{\mu_{b-1}+1}}\cdots
\widetilde{A}_{s_{\mu_b}})A_0\cdot(\widetilde{A}_{s_{\mu_b+1}}\cdots
\widetilde{A}_{s_c})\Big)
\end{equation*}
 or, respectively,
\begin{equation}\label{form of elements}
tr\Big((\widetilde{A}_{s_1}\cdots
\widetilde{A}_{s_{\mu_1}})A_0\cdot(\widetilde{A}_{s_{\mu_1+1}}\cdots\widetilde{A}_{s_{\mu_2}})
A_0 \cdot\cdots\cdot(\widetilde{A}_{s_{\mu_{b-1}+1}}\cdots
\widetilde{A}_{s_{c}})A_0\Big).
\end{equation}
Then we can define
\begin{equation}\label{Omega-ab}
\Omega_{ir}^c(a,b):=\sum_{\mu\in\Lambda^c_b}\sum_{\omega\in\mathscr{T}_{ir}(a,\mu)}\omega,\quad
\overline{\Omega}_{ir}^c(a,b):=\sum_{\mu\in\overline{\Lambda}^c_b}\sum_{\omega\in\mathscr{T}_{ir}(a,\mu)}\omega,
\end{equation}
which are essentially identical when $c>b$ (the first is zero when
$c=b$) since, by the symmetry of the \emph{trace} function, both are
the sum of all non-vanishing elements $tr(A_{\tau_1}\cdots
A_{\tau_i})$, $\tau\in\mathscr{S}_i(\sigma)$,
$\sigma\in\mathscr{P}_{ir}$ with $\sigma_0=a$ (without counting
multiplicities $\Theta,\overline{\Theta}$), of the form (\ref{form
of elements}) with $b$ copies of $A_0$ occurring separately among
$c$ number of $\widetilde{A}_{s}$ where $\widetilde{A}_{s}=A_{s}$ or
$A_sA_1^{\iota_s}A_{s'}$ ($s,s'\geq2$, $\iota_s\geq1$).

Equipped with these refinements, we are ready to derive a more
tractable formula than (\ref{Upsilon ir}) for the coefficients
$\Upsilon_{ir}$ in the power series expression of the $i$-th order
mean curvature $Q_i(t)$ of the tubular hypersurface $M_t^n$.
\begin{prop}
With notations as above, we have for $i\geq2,~r\geq2$,
\begin{equation}\label{Upsilon i,r}
\Upsilon_{ir}=\sum_{\omega\in\mathscr{T}_{ir}(0)}\omega+\sum_{c=1}^{D(i,r)}\sum_{
a\in\mathcal {A}_{ir}^c}(-1)^{ a}\sum_{b=1}^c
\Big(\Theta(a,b)\Omega_{ir}^{c}(a,b)+\overline{\Theta}(
a,b)\overline{\Omega}_{ir}^{c}(a,b)\Big).
\end{equation}
\end{prop}
\begin{rem} As footnoted before, the index $b$ in the summation actually
takes values from $1$ to $\min\{c,d(i,r)\}$, though it does not
matter for the calculation since when $c\geq b>d(i,r)$, $\mathcal
{A}_{ir}^c\ni a\leq i-c\leq d(i,r)<b$ and thus
$\Theta(a,b)=\overline{\Theta}(a,b)=0$.
\end{rem}
\begin{rem}\label{low order ex Upsilon}
For example, we list some low order cases as follows:
\begin{eqnarray}
&&\Upsilon_{22}=tr(A_1^2)-2tr(A_2A_0),~~
\Upsilon_{i2}=(-1)^{i-1}i~tr(A_2A_0),~~ for~~i\geq3;\nonumber\\
&& \Upsilon_{23}=2tr(A_1A_2)-2tr(A_3A_0),~~
\Upsilon_{33}=tr(A_1^3)+3tr(A_3A_0),\nonumber\\
&&\Upsilon_{i3}=(-1)^{i-1}i~tr(A_3A_0),~~ for~~ i\geq4;\nonumber\\
&&\Upsilon_{24}=2tr(A_1A_3)+tr(A_2^2)-2tr(A_4A_0),~~\Upsilon_{34}=3tr(A_1^2A_2)+3tr(A_4A_0)-3tr(A_2^2A_0),\nonumber\\
&&\Upsilon_{44}=tr(A_1^4)-4tr(A_4A_0)+4tr(A_2^2A_0)+2tr(A_2A_0A_2A_0),\nonumber\\
&&\Upsilon_{i4}=(-1)^{i-1}i~tr(A_4A_0)
+(-1)^{i-2}\Big(i~tr(A_2^2A_0)+\frac{i(i-3)}{2}tr(A_2A_0A_2A_0)\Big),~~for~~i\geq5.\nonumber
\end{eqnarray}
\end{rem}
\begin{proof}
Based on the refinements above, direct calculations show
\begin{eqnarray}
\Upsilon_{ir}&=&\sum_{ a=0}^{i-1}(-1)^{
a}\sum_{\omega\in\mathscr{T}_{ir}( a)}\omega
=\sum_{\omega\in\mathscr{T}_{ir}(0)}\omega+\sum_{ a=1}^{i-1}(-1)^{ a}\sum_{c=1}^{D(i,r)}\sum_{\lambda\in\Lambda^c( a)}\sum_{\omega\in\mathscr{T}_{ir}(\lambda)}\omega \nonumber\\
&=&\sum_{\omega\in\mathscr{T}_{ir}(0)}\omega+\sum_{c=1}^{D(i,r)}\sum_{
a\in\mathcal {A}_{ir}^c}(-1)^{
a}\sum_{b=1}^c\Big(\sum_{\mu\in\Lambda^c_b}+\sum_{\mu\in\overline{\Lambda}^c_b}\Big)
\sum_{\lambda\in\Lambda^c( a,\mu)}\sum_{\omega\in\mathscr{T}_{ir}(\lambda)}\omega \nonumber\\
&=&\sum_{\omega\in\mathscr{T}_{ir}(0)}\omega+\sum_{c=1}^{D(i,r)}\sum_{
a\in\mathcal {A}_{ir}^c}(-1)^{
a}\sum_{b=1}^c\Big(\sum_{\mu\in\Lambda^c_b}+\sum_{\mu\in\overline{\Lambda}^c_b}\Big)
\sum_{\lambda\in\Lambda^c( a,\mu)}
\sum_{\omega\in\mathscr{T}_{ir}(a,\mu)}\omega\nonumber\\
&=&\sum_{\omega\in\mathscr{T}_{ir}(0)}\omega+\sum_{c=1}^{D(i,r)}\sum_{
a\in\mathcal {A}_{ir}^c}(-1)^{ a}\sum_{b=1}^c
\Big(\sum_{\mu\in\Lambda^c_b}\Theta(
a,b)+\sum_{\mu\in\overline{\Lambda}^c_b}\overline{\Theta}( a,b)\Big)
\sum_{\omega\in\mathscr{T}_{ir}(a,\mu)}\omega\nonumber\\
&=&\sum_{\omega\in\mathscr{T}_{ir}(0)}\omega+\sum_{c=1}^{D(i,r)}\sum_{
a\in\mathcal {A}_{ir}^c}(-1)^{ a}\sum_{b=1}^c
\Big(\Theta(a,b)\Omega_{ir}^{c}(a,b)+\overline{\Theta}(
a,b)\overline{\Omega}_{ir}^{c}(a,b)\Big).\nonumber
\end{eqnarray}
The examples in Remark \ref{low order ex Upsilon} can be verified by
(\ref{Upsilon i,r}) immediately.
\end{proof}

Now we consider the cases when $i\geq r\geq2$. Obviously, we have
now
$$d(i,r)=D(i,r)=[\frac{r}{2}]=:d.$$ It is easily seen from the
definitions that $\mathscr{T}_{rr}(0)$ consists of only one element
$tr(A_1^r)$ and $\mathscr{T}_{ir}(0)$ is empty for $i>r$. Moreover,
for each $e\geq1$, the map from $\mathscr{P}_{rr}$ to
$\mathscr{P}_{r+e~r}$ defined by
$$(\sigma_0,\sigma_1,\cdots,\sigma_r)\mapsto
(\sigma_0+e,\sigma_1,\cdots,\sigma_r)$$ gives a one-to-one
correspondence. Consequently, we have
\begin{eqnarray}
&&\mathcal {A}_{r+e~r}^c=\{a+e~|~a\in\mathcal {A}_{rr}^c\},\quad
for~~ 1\leq c\leq d,\nonumber\\
&&\mathscr{T}_{r+e~r}(a+e,\mu)=\mathscr{T}_{rr}(a,\mu),\quad for
~~a\in\mathcal {A}_{rr}^c,~~ \mu\in\Lambda^c_b~ or
~\overline{\Lambda}^c_b,\nonumber
\end{eqnarray}
 and thus
 \begin{equation*}
\Omega_{r+e~r}^c(a+e,b)=\Omega_{rr}^c(a,b),~~
\overline{\Omega}_{r+e~r}^c(a+e,b)=\overline{\Omega}_{rr}^c(a,b),\quad
for ~~a\in\mathcal {A}_{rr}^c,~~1\leq b\leq c.
\end{equation*}
Therefore, the formulae for $\Upsilon_{rr}$ and $\Upsilon_{r+e~r}$
($r\geq2,~e\geq1$) in the form (\ref{Upsilon i,r}) can be rewritten
as
\begin{eqnarray}
&&\Upsilon_{rr}=tr(A_1^r)+\sum_{c=1}^d\sum_{a\in\mathcal
{A}_{rr}^c}(-1)^{ a}\sum_{b=1}^c
\Big(\Theta(a,b)\Omega_{rr}^c(a,b)+\overline{\Theta}(
a,b)\overline{\Omega}_{rr}^c(a,b)\Big),\label{Upsilon rr}\\
&&\Upsilon_{r+e~r}=\sum_{c=1}^d\sum_{a\in\mathcal {A}_{rr}^c}(-1)^{
a+e}\sum_{b=1}^c
\Big(\Theta(a+e,b)\Omega_{rr}^c(a,b)+\overline{\Theta}(
a+e,b)\overline{\Omega}_{rr}^c(a,b)\Big).\label{Upsilon re r}
\end{eqnarray}

Similarly, for $r\geq3$, $\mathscr{T}_{r-1~r}(0)$ consists of
$(r-1)$ copies of $tr(A_1^{r-2}A_2)$. Moreover, we have
\begin{eqnarray}
&&\mathcal {A}_{r-1~r}^c=\{a-1~|~a\in\mathcal {A}_{rr}^c\},\quad
for~~ 1\leq c\leq d,\nonumber\\
&&\mathscr{T}_{r-1~r}(a-1,\mu)=\mathscr{T}_{rr}(a,\mu),\quad for
~~a\in\mathcal {A}_{rr}^c, ~~a-1\geq b,~~ \mu\in\Lambda^c_b~ or
~\overline{\Lambda}^c_b,\nonumber
\end{eqnarray}
 and thus
 \begin{equation*}
\Omega_{r-1~r}^c(a-1,b)=\Omega_{rr}^c(a,b),~~
\overline{\Omega}_{r-1~r}^c(a-1,b)=\overline{\Omega}_{rr}^c(a,b),\quad
for ~~a\in\mathcal {A}_{rr}^c, ~~a-1\geq b,~~1\leq b\leq c.
\end{equation*}
Therefore, the formula for $\Upsilon_{r-1~r}$ ($r\geq3$) can be
rewritten as
\begin{equation}\label{Upsilon r-1}
\Upsilon_{r-1~r}=(r-1)tr(A_1^{r-2}A_2)+\sum_{c=1}^d\sum_{a\in\mathcal
{A}_{rr}^c}(-1)^{a-1}\sum_{b=1}^c
\Big(\Theta(a-1,b)\Omega_{rr}^c(a,b)+\overline{\Theta}(
a-1,b)\overline{\Omega}_{rr}^c(a,b)\Big),
\end{equation}
which also holds for $r=2$ by (\ref{Upsilon 0,1}) and the initial
conditions of $\Theta$, $\overline{\Theta}$.

Taking iterative sum of (\ref{inductive Theta}), we obtain the
following useful formula for $\Theta$ and $\overline{\Theta}$:
\begin{equation}\label{iterative sum of theta}
C_e^0\Theta(a,b)-C_e^1\Theta(a+1,b)+\cdots+(-1)^eC_e^e\Theta(a+e,b)=(-1)^e\Theta(a,b-e),
\end{equation}
for any $a,b\in\mathbb{Z}$ and $e\geq0$, where
$C_e^j=\frac{e!}{j!(e-j)!}$. For $r\geq2,~e\geq0$, put
\begin{eqnarray}
\Phi_r(e):=C_e^0\Upsilon_{rr}+C_e^1\Upsilon_{r+1~r}+\cdots+C_e^e\Upsilon_{r+e~r},\nonumber\\
\Psi_r(e):=C_e^0\Upsilon_{r-1~r}+C_e^1\Upsilon_{rr}+\cdots+C_e^e\Upsilon_{r-1+e~r},\nonumber
\end{eqnarray}
then taking sum of (\ref{Upsilon rr}), (\ref{Upsilon re r}) and
(\ref{Upsilon r-1}) iteratively by using (\ref{iterative sum of
theta}), we obtain
\begin{eqnarray}
&&\Phi_r(e)=tr(A_1^r)+\sum_{c=1}^d\sum_{a\in\mathcal
{A}_{rr}^c}(-1)^{ a+e}\sum_{b=1}^c
\Big(\Theta(a,b-e)\Omega_{rr}^c(a,b)+\overline{\Theta}(
a,b-e)\overline{\Omega}_{rr}^c(a,b)\Big),\nonumber\\
&&\Psi_r(e)=(r-1)tr(A_1^{r-2}A_2)+e~tr(A_1^r)\nonumber\\&&\quad\quad\quad+\sum_{c=1}^d\sum_{a\in\mathcal
{A}_{rr}^c}(-1)^{ a-1+e}\sum_{b=1}^c
\Big(\Theta(a-1,b-e)\Omega_{rr}^c(a,b)+\overline{\Theta}(
a-1,b-e)\overline{\Omega}_{rr}^c(a,b)\Big).\nonumber
\end{eqnarray}
In particular, since $\mathcal {A}_{rr}^d=\{r-d\}$, by the initial
conditions of $\Theta$, $\overline{\Theta}$, we have
\begin{eqnarray}
&&\Phi_r(d)=tr(A_1^r)+(-1)^r\overline{\Omega}_{rr}^d(r-d,d),\nonumber\\
&&\Phi_r(d+1)=tr(A_1^r),\nonumber\\
&&\Psi_r(d)=(r-1)tr(A_1^{r-2}A_2)+d~tr(A_1^r)+(-1)^{r-1}\overline{\Omega}_{rr}^d(r-d,d),\nonumber\\
&&\Psi_r(d+1)=(r-1)tr(A_1^{r-2}A_2)+(d+1)~tr(A_1^r),\nonumber
\end{eqnarray}
where by (\ref{Omega-ab}),
\begin{equation*}
\overline{\Omega}_{rr}^d(r-d,d)=\begin{cases} tr\Big((A_2A_0)^d\Big),&for~~r=2d;\\
d~tr\Big((A_2A_0)^{d-1}A_3A_0\Big),& for~~r=2d+1.
\end{cases}
\end{equation*}
Recall the formulae of $A_r$ in (\ref{shape bar expansion}), then
the above formulae can be rewritten as ($d\geq1$)
\begin{eqnarray}
&&\Phi_{2d}(d)=tr\Big(S_{\nu}^{2d}\Big)+3^{-d}tr\Big((K_{\nu}^{\bot})^d\Big),\nonumber\\
&&\Phi_{2d+1}(d)=tr\Big(S_{\nu}^{2d+1}\Big)-d~3^{-d+1}4^{-1}tr\Big((K_{\nu}^{\bot})^{d-1}\mathcal {K}_{\nu}^{\bot}\Big),\nonumber\\
&&\Phi_{2d}(d+1)=tr\Big(S_{\nu}^{2d}\Big), \nonumber\\
&&\Phi_{2d+1}(d+1)=tr\Big(S_{\nu}^{2d+1}\Big),\label{sequence phi psi}\\
&&\Psi_{2d}(d)=(2d-1)tr\Big(S_{\nu}^{2d-2}K_{\nu}^{\top}\Big)+(3d-1)~tr\Big(S_{\nu}^{2d}\Big)-3^{-d}tr\Big((K_{\nu}^{\bot})^d\Big),\nonumber\\
&&\Psi_{2d+1}(d)=(2d)tr\Big(S_{\nu}^{2d-1}K_{\nu}^{\top}\Big)+(3d)~tr\Big(S_{\nu}^{2d+1}\Big)+d~3^{-d+1}4^{-1}tr\Big((K_{\nu}^{\bot})^{d-1}\mathcal {K}_{\nu}^{\bot}\Big),\nonumber\\
&&\Psi_{2d}(d+1)=(2d-1)tr\Big(S_{\nu}^{2d-2}K_{\nu}^{\top}\Big)+(3d)~tr\Big(S_{\nu}^{2d}\Big),\nonumber\\
&&\Psi_{2d+1}(d+1)=(2d)tr\Big(S_{\nu}^{2d-1}K_{\nu}^{\top}\Big)+(3d+1)~tr\Big(S_{\nu}^{2d+1}\Big).\nonumber
\end{eqnarray}
In conclusion, we get the following
\begin{thm}\label{Thm-subm-d}
Let $M^m$ be a submanifold of a Riemannian manifold $N^{n+1}$ and
$M_t^n$ be the tubular hypersurface around $M$ of sufficiently small
radius $t\in(0,\varepsilon)$. For any integer $d\geq1$,
\begin{itemize}
\item[(i)] if each $M_t^n$ has constant
$Q_{2d},Q_{2d+1},\cdots,Q_{3d}$, then on $M$,
$$Q_{2d}^{\nu}+3^{-d}\rho_d(K_{\nu}^{\bot})\equiv Const;$$
\item[(ii)] if each $M_t^n$ has constant
$Q_{2d+1},Q_{2d+2},\cdots,Q_{3d+1}$, then on $M$,
$$Q_{2d+1}^{\nu}-d~3^{-d+1}4^{-1}tr\Big((K_{\nu}^{\bot})^{d-1}\mathcal
{K}_{\nu}^{\bot}\Big)\equiv 0;$$
\item[(iii)] if each $M_t^n$ has constant
$Q_{2d},Q_{2d+1},\cdots,Q_{3d+1}$, then on $M$, $$Q_{2d}^{\nu}\equiv
Const;$$
\item[(iv)] if each $M_t^n$ has constant
$Q_{2d+1},Q_{2d+2},\cdots,Q_{3d+2}$, then on $M$,
$$Q_{2d+1}^{\nu}\equiv 0;$$
\item[(v)] if each $M_t^n$ has constant
$Q_{2d-1},Q_{2d},\cdots,Q_{3d-1}$, then on $M$,
$$(2d-1)tr\Big(S_{\nu}^{2d-2}K_{\nu}^{\top}\Big)+(3d-1)~Q_{2d}^{\nu}-3^{-d}\rho_d(K_{\nu}^{\bot})\equiv Const;$$
\item[(vi)] if each $M_t^n$ has constant
$Q_{2d},Q_{2d+1},\cdots,Q_{3d}$, then on $M$,
$$(2d)tr\Big(S_{\nu}^{2d-1}K_{\nu}^{\top}\Big)+(3d)~Q_{2d+1}^{\nu}+d~3^{-d+1}4^{-1}tr\Big((K_{\nu}^{\bot})^{d-1}\mathcal
{K}_{\nu}^{\bot}\Big)\equiv0;$$
\item[(vii)] if each $M_t^n$ has constant
$Q_{2d-1},Q_{2d},\cdots,Q_{3d}$, then on $M$,
$$(2d-1)tr\Big(S_{\nu}^{2d-2}K_{\nu}^{\top}\Big)+(3d)~Q_{2d}^{\nu}\equiv Const;$$
\item[(viii)] if each $M_t^n$ has constant
$Q_{2d},Q_{2d+1},\cdots,Q_{3d+1}$, then on $M$,
$$(2d)tr\Big(S_{\nu}^{2d-1}K_{\nu}^{\top}\Big)+(3d+1)~Q_{2d+1}^{\nu}\equiv 0.$$
\end{itemize}
\end{thm}
\begin{proof}
Recall the definition of $\Upsilon_{ir}$ which is the coefficient of
$t^r$ defined over the unit normal bundle $\mathcal {V}_1M$ of $M$
in the power series expansion of $\Upsilon_i(t):=t^iQ_i(t)$ with
respect to $t\in(0,\varepsilon)$, where $Q_i(t)$ is the $i$-th order
mean curvature of $M_t^n$. Therefore, if $Q_i(t)$ is constant on
$M_t^n$ and thus is a function depending only on $t$, then
$\Upsilon_{ir}$ would be constant on $\mathcal {V}_1M$ for each
$r\geq0$, which implies that each $\Phi_r(e),~\Psi_r(e)$
($r=2d,2d+1,~e=d,d+1$) in the sequence (\ref{sequence phi psi})
would be constant under the corresponding assumptions listed in the
proposition. This verifies the identities in the cases where the
vanishing assertion in (ii),(iv),(vi),(viii) is because of the
anti-symmetry with respect to $\nu$ of
$Q_{2d+1}^{\nu},~tr\Big((K_{\nu}^{\bot})^{d-1}\mathcal
{K}_{\nu}^{\bot}\Big)$ and
$tr\Big(S_{\nu}^{2d-1}K_{\nu}^{\top}\Big).$
\end{proof}

Now we are ready to prove Theorem \ref{Thm-subm-l}. \vskip 0.2cm
 \textbf{Proof of Theorem \ref{Thm-subm-l}.} The case (b) follows
 directly from Theorem \ref{Thm-subm-d} with $d=[\frac{l}{2}]$. It
 suffices to prove the case (a) for $1\leq l\leq4$.

For $l=1$, it follows from (\ref{Upsilon 0,1}) that
$Q_1^{\nu}=tr(A_1)=\Upsilon_{11}$ which is the coefficient of $t^1$
in the power series expansion of $\Upsilon_1(t):=tQ_1(t)$ with
respect to $t\in(0,\varepsilon)$. Therefore, if $Q_1(t)$ is constant
on each tubular hypersurface $M_t^n$, then $Q_1^{\nu}$ is constant
on the unit normal bundle of $M$ and thus vanishes since
$Q_1^{-\nu}=-Q_1^{\nu}$.

For $l=2$, if $Q_2(t),Q_3(t)$ are constant on each tubular
hypersurface $M_t^n$, then as the coefficients of $t^2$ in the power
series expansions of $t^2Q_2(t)$ and $t^3Q_3(t)$, $\Upsilon_{22}$
and $\Upsilon_{32}$ would be constant. So in this case, by Remark
\ref{low order ex Upsilon}, we get
$$\Upsilon_{22}=Q_2^{\nu}-\frac{2}{3}tr(K_{\nu}^{\bot})\equiv
Const,\quad \Upsilon_{32}=tr(K_{\nu}^{\bot})\equiv Const,$$ which
verify the first assertion. If in addition $Q_1(t)$ is constant on
$M_t^n$, then by (\ref{Upsilon 0,1}),
$$\Upsilon_{12}=tr(A_2)=Q_2^{\nu}+tr(K_{\nu}^{\top})+\frac{1}{3}tr(K_{\nu}^{\bot})\equiv
Const,$$ which implies $tr(K_{\nu}^{\top})\equiv Const$ and thus
verifies the second assertion.

For $l=3$, if $Q_3(t),Q_4(t)$ are constant on each tubular
hypersurface $M_t^n$, then as the coefficients of $t^3$ in the power
series expansions of $t^3Q_3(t)$ and $t^4Q_4(t)$, $\Upsilon_{33}$
and $\Upsilon_{43}$ would be constant. So in this case, by Remark
\ref{low order ex Upsilon} and the anti-symmetry of $Q_3^{\nu}$ and
$tr(\mathcal {K}_{\nu}^{\bot})$ with respect to $\nu$, we get
$$\Upsilon_{33}=Q_3^{\nu}+\frac{3}{4}tr(\mathcal
{K}_{\nu}^{\bot})\equiv0,\quad \Upsilon_{43}=-tr(\mathcal
{K}_{\nu}^{\bot})\equiv0,$$which verify the first assertion. If in
addition $Q_2(t)$ is constant on $M_t^n$, then by Remark \ref{low
order ex Upsilon} and the anti-symmetry of $Q_3^{\nu}$ and
$tr(S_{\nu}K_{\nu}^{\top})$ with respect to $\nu$,
$$\Upsilon_{23}=2Q_3^{\nu}+2tr(S_{\nu}K_{\nu}^{\top})-\frac{1}{2}tr(\mathcal
{K}_{\nu}^{\bot})\equiv0,$$ which implies
$tr(S_{\nu}K_{\nu}^{\top})\equiv0$ and thus verifies the second
assertion.

For $l=4$, if $Q_4(t), Q_5(t), Q_6(t)$ are constant on each tubular
hypersurface $M_t^n$, then as the coefficients of $t^4$ in the power
series expansions of $t^4Q_4(t)$, $t^5Q_5(t)$ and $t^6Q_6(t)$,
$\Upsilon_{44}$, $\Upsilon_{54}$ and $\Upsilon_{64}$ would be
constant.  So in this case, by Remark \ref{low order ex Upsilon}, we
get
\begin{eqnarray}
&&\Upsilon_{44}=Q_4^{\nu}-4\Big(tr(A_4A_0)-tr(A_2^2A_0)\Big)+\frac{2}{9}\rho_2(K_{\nu}^{\bot})\equiv
Const,\nonumber\\
&&\Upsilon_{54}=5\Big(tr(A_4A_0)-tr(A_2^2A_0)\Big)-\frac{5}{9}\rho_2(K_{\nu}^{\bot})\equiv
Const,\nonumber\\
&&\Upsilon_{64}=-6\Big(tr(A_4A_0)-tr(A_2^2A_0)\Big)+\rho_2(K_{\nu}^{\bot})\equiv
Const,\nonumber
\end{eqnarray}
which imply that $Q_4^{\nu}\equiv Const$,
$\rho_2(K_{\nu}^{\bot})\equiv Const$, $tr(A_4A_0)-tr(A_2^2A_0)\equiv
Const$ and thus verify the first assertion. If in addition $Q_3(t)$
is constant on $M_t^n$, then by Remark \ref{low order ex Upsilon},
$$\Upsilon_{34}=3Q_4^{\nu}+3tr(S_{\nu}^2K_{\nu}^{\top})+3\Big(tr(A_4A_0)-tr(A_2^2A_0)\Big)\equiv
Const,$$ which implies $tr(S_{\nu}^2K_{\nu}^{\top})\equiv Const$ and
thus verifies the second assertion.

The proof is now complete. \hfill $\Box$

\section{Focal submanifolds}
This section is devoted to the proof of Theorem \ref{Thm-focal-filt}
which is a geometrical filtration for the focal submanifolds of
isoparametric functions on a complete Riemannian manifold. The
assertions in (i-ix) of this theorem essentially come from Theorem
\ref{Thm-subm-l} and Theorem \ref{Thm-subm-d}, while (a-d) of this
theorem treat some special cases as corollaries of (i-ix) and the
following preliminary on austere submanifolds.

\begin{prop}\label{austere CPC}
Let $M^m$ be an austere submanifold of constant principal curvatures
in a Riemannian manifold $N^{n+1}$. If $m=2,~n\geq4$; or $m=3,~
n\geq5$; or $m=4,~n\geq10$, then $M$ is a totally geodesic
submanifold in $N$.
\end{prop}
\begin{proof}
Recall that a submanifold $M^m$ is called an austere submanifold of
constant principal curvatures if there exist some constants
$\lambda_1,\cdots,\lambda_p$ such that the shape operator $S_{\nu}$
of $M$ with respect to any unit normal vector $\nu$ at any point has
eigenvalues
$\lambda_1,-\lambda_1,\cdots,\lambda_p,-\lambda_p,0,\cdots,0$
($(m-2p)$ zeroes). In particular, on such submanifold we have
$$\|S_{\nu}\|^2\equiv Const=:C,$$
which implies that for any orthonormal frame
$\{e_{m+1},\cdots,e_{n+1}\}$ of the normal bundle $\mathcal {V}M$ of
$M$, $$\langle S_{e_i}, S_{e_j}\rangle=C\delta_{ij}, \quad
for~~i,j=m+1,\cdots,n+1.$$ Therefore, if $M$ is not totally geodesic
and thus $C>0$, then $S_{e_{m+1}},\cdots,S_{e_{n+1}}$ are
independent self-dual operators on the tangent bundle $\mathcal
{T}M$ of $M$ with constant eigenvalues of opposite signs, which
means that at any point $q$ of $M$ the space $\mathcal {S}_q$ of
shape operators $S_{\nu}$, $\nu\in\mathcal {V}_qM$, is an
$(n+1-m)$-dimensional subspace of self-dual operators on $\mathcal
{T}_qM$.

If the shape operators are looked as quadratic functions on
$\mathcal {T}_qM$ via the metric, then $\mathcal {S}_q$ is an
$(n+1-m)$-dimensional \emph{austere} subspace of quadratic functions
on $\mathcal {T}_qM$ in the sense of \cite{Bry} where, among other
things, Bryant solved the classification problem of maximal austere
subspaces of quadratic functions on a real vector space of dimension
$m=2,3,~or~4$. In particular, it follows from his classification
that each maximal austere subspace is of dimension $2$ when $m=2$ or
$3$, and not greater than $6$ when $m=4$. Consequently,
$n+1-m=dim(\mathcal {S}_q)\leq2$ when $m=2$ or $3$, and
$n+1-m=dim(\mathcal {S}_q)\leq6$ when $m=4$, which verifies the
assertions by contradiction.
\end{proof}
One can see from the proof above that for fixed $m$ and sufficiently
large $n$, an austere submanifold $M^m$ of constant principal
curvatures in $N^{n+1}$ should be totally geodesic. It is
interesting to find out an optimal relationship for such pairs of
$(m,n)$.

 \textbf{Proof of Theorem \ref{Thm-focal-filt}.} First of all, by
 definition $f$ is a $k$-isoparametric function if and only if each
 regular level hypersurface $M_t^n$ of $f$ has constant higher order mean
 curvatures $Q_1(t),\cdots,Q_k(t)$. As showed in section \ref{section 4}, if $f$
 is a $k$-isoparametric function, then the coefficients $\Upsilon_{ir}$ in the power series
  expansion formula of $t^iQ_i(t)$ are constant for $1\leq i\leq k$,
  $r\geq0$.

Now we come to verify the assertions listed in the theorem case by
case.

(i) $k=1$. By (\ref{Upsilon 0,1}) and (\ref{shape bar expansion}),
\begin{eqnarray}
&&\Upsilon_{11}=tr(A_1)=tr(S_{\nu})\equiv Const,\nonumber\\
&&\Upsilon_{12}=tr(A_2)=Q_2^{\nu}+tr(K_{\nu}^{\top})+\frac{1}{3}tr(K_{\nu}^{\bot})\equiv
Const,\nonumber\\
&&\Upsilon_{13}=tr(A_3)=Q_3^{\nu}+tr(S_{\nu}K_{\nu}^{\top})+\frac{1}{2}tr(\mathcal
{K}_{\nu}^{\top})+\frac{1}{4}tr(\mathcal {K}_{\nu}^{\bot})\equiv
Const.\nonumber
\end{eqnarray}
Moreover, since $S_{-\nu}=-S_{\nu}$,
$K_{-\nu}^{\top}=K_{\nu}^{\top}$, $\mathcal
{K}_{-\nu}^{\top}=-\mathcal {K}_{\nu}^{\top}$ and $\mathcal
{K}_{-\nu}^{\bot}=-\mathcal {K}_{\nu}^{\bot}$, we know that
$\Upsilon_{11}$ and $\Upsilon_{13}$ are anti-symmetric with respect
to $\nu$ and thus $\Upsilon_{11}=\Upsilon_{13}\equiv 0$.

(ii) $k=2$. As in the proof of Theorem \ref{Thm-subm-l}, by Remark
\ref{low order ex Upsilon} we can get
\begin{eqnarray}
&&\Upsilon_{22}=Q_2^{\nu}-\frac{2}{3}tr(K_{\nu}^{\bot})\equiv
Const,\nonumber\\
&&
\frac{1}{2}\Upsilon_{23}=Q_3^{\nu}+tr(S_{\nu}K_{\nu}^{\top})-\frac{1}{4}tr(\mathcal
{K}_{\nu}^{\bot})\equiv0.\nonumber
\end{eqnarray}
Combining these with (i), we get
\begin{eqnarray}
&&tr(K_{\nu})=tr(K_{\nu}^{\top})+tr(K_{\nu}^{\bot})=\Upsilon_{12}-\Upsilon_{22}\equiv Const,\nonumber\\
&&\frac{1}{2}tr(\mathcal {K}_{\nu})=\frac{1}{2}tr(\mathcal
{K}_{\nu}^{\top})+\frac{1}{2}tr(\mathcal
{K}_{\nu}^{\bot})=\Upsilon_{13}-\frac{1}{2}\Upsilon_{23}\equiv0.\nonumber
\end{eqnarray}

(iii) $k=3$. Similarly as before we can get
\begin{eqnarray}
&&\Upsilon_{32}=tr(K_{\nu}^{\bot})\equiv
Const,\nonumber\\
&&\Upsilon_{33}=Q_3^{\nu}+\frac{3}{4}tr(\mathcal
{K}_{\nu}^{\bot})\equiv0.\nonumber
\end{eqnarray}
Combining these with (ii), we get
\begin{eqnarray}
&&Q_2^{\nu}=\Upsilon_{22}+\frac{2}{3}\Upsilon_{32}\equiv Const,\nonumber\\
&&tr(S_{\nu}K_{\nu}^{\top})-tr(\mathcal
{K}_{\nu}^{\bot})=\frac{1}{2}\Upsilon_{23}-\Upsilon_{33}\equiv0.\nonumber
\end{eqnarray}

(iv) $k=4$. Similarly as before we can get
$$\Upsilon_{43}=-tr(\mathcal {K}_{\nu}^{\bot})\equiv0,$$
which implies
$Q_3^{\nu}=\Upsilon_{33}+\frac{3}{4}\Upsilon_{43}\equiv0,$
$tr(S_{\nu}K_{\nu}^{\top})=\frac{1}{2}\Upsilon_{23}-\Upsilon_{33}-\Upsilon_{43}\equiv0.$

(v) $k=5$. Similarly as before we can get
\begin{eqnarray}
&&Q_{4}^{\nu}-\frac{2}{9}\rho_2(K_{\nu}^{\bot})=\Upsilon_{44}+\frac{4}{5}\Upsilon_{54}\equiv
Const,\nonumber\\
&&3tr\Big(S_{\nu}^{2}K_{\nu}^{\top}\Big)+\rho_2(K_{\nu}^{\bot})=\Upsilon_{34}-3\Upsilon_{44}-3\Upsilon_{54}\equiv
Const.\nonumber
\end{eqnarray}

(vi) $k=6$. Similarly as before we can get
$$\rho_2(K_{\nu}^{\bot})=18\Big(\frac{1}{5}\Upsilon_{54}+\frac{1}{6}\Upsilon_{64}\Big)\equiv
Const,$$ which, together with (v), implies $Q_{4}^{\nu}\equiv
Const,$ $tr\Big(S_{\nu}^{2}K_{\nu}^{\top}\Big)\equiv Const.$ By (vi)
of Theorem \ref{Thm-subm-d}, we can get
$2tr\Big(S_{\nu}^{3}K_{\nu}^{\top}\Big)+3Q_{5}^{\nu}+\frac{1}{12}tr\Big(K_{\nu}^{\bot}\mathcal
{K}_{\nu}^{\bot}\Big)\equiv0.$

(vii) $k=3d+1$, $d\geq2$. The identities listed in this case
successively come from (iii),(i),(v),(ii),(viii) of Theorem
\ref{Thm-subm-d} since now $M_t^n$ has constant
$Q_1,\cdots,Q_{3d+1}$.

(viii) $k=3d+2$, $d\geq2$. The identities listed in this case
successively come from (iv),(viii),(vi),(v) of Theorem
\ref{Thm-subm-d} since now $M_t^n$ has constant
$Q_1,\cdots,Q_{3d+2}$.

(ix) $k=3d+3$, $d\geq2$. The identities listed in this case
successively come from (i),(vi) of Theorem \ref{Thm-subm-d} since
now $M_t^n$ has constant $Q_1,\cdots,Q_{3d+3}$.

(a) $m=0$. Obviously, $Q_2^{\nu}=tr(K_{\nu}^{\top})\equiv0$, which,
together with the second identity in (i), implies
$Ric(\nu)=tr(K_{\nu})=tr(K_{\nu}^{\bot})\equiv Const$ for any
$k\geq1$.

(b) $m=n$. By continuity it is easily seen that each higher order
mean curvature $Q_i(t)$ of $M_t^n$ converges to $Q_i$ of $M$ when
$t$ goes to $0$, \emph{i.e.}
$Q_i=\lim\limits_{t\rightarrow0}Q_i(t).$ Therefore, if $f$ is
$k$-isoparametric, then $Q_i(t)$ and hence $Q_i$ are constant
functions on $M_t^n$ and $M$ respectively, which shows that $M$ is
$k$-isoparametric since (locally) $M_t^n$ consists of two
equidistant parallel hypersurfaces, say $M_t^{+}$ and $M_t^{-}$, on
both ``sides" of $M$. The odd order mean curvatures $Q_{2j+1}$
$(2j+1\leq k)$ vanish on $M$ because of the anti-symmetry of odd
order mean curvatures with respect to unit normal vectors and the
identically constancy of $Q_{2j+1}(t)$ on $M_t^{+}$ and $M_t^{-}$.

(c) For $m\leq[\frac{2k+1}{3}]$, $k\leq6$, (i-vi) above show that
$M^m$ has constant $Q_1^{\nu},\cdots,Q_m^{\nu}$ and thus has
constant principal curvatures which occur in opposite signs since
all odd order mean curvatures vanish. For $m\leq[\frac{2k-1}{3}]$,
$k\geq7$, (vii-ix) derive the same conclusion as above. For $k=n$,
\emph{i.e.}, $f$ is totally isoparametric, then each order mean
curvature of $M_t^n$ is constant, which by Theorem \ref{Thm-subm-l}
implies that $M$ has constant each order mean curvature and thus is
an austere submanifold of constant principal curvatures. The second
part of this case has been proved in Proposition \ref{austere CPC}.

(d) Similarly as in (c), it is easily seen that under each
assumption of $m$ and $k$,
$\rho_1(K_{\nu}^{\bot}),\cdots,\rho_{n-m}(K_{\nu}^{\bot})$ are
constant on $M$ and thus the (restricted) vertical Jacobi operator
$K_{\nu}^{\bot}$ has constant eigenvalues since the restricted
vertical Jacobi operator $K_{\nu}^{\bot}$ is a self-dual operator of
order $(n-m)$.

The proof is now complete. \hfill $\Box$

\begin{ack}
This paper is greatly indebted to a previous joint work \cite{GTY}
with Professor Zizhou Tang and Doctor Wenjiao Yan during which I
enjoyed and benefited much from Tang's instruction and Yan's
friendship. By this opportunity I would like to express my sincere
gratitude to them. Many thanks also to Professors Q. S. Chi, C. K.
Peng and G. Thorbergsson for their valuable suggestions and helpful
comments during the preparation of this paper.
\end{ack}

\end{document}